\documentclass[10pt,a4paper, reqno ]{amsart}
\usepackage[foot]{amsaddr}
\usepackage[toc,page]{appendix}
\usepackage[utf8]{inputenc}
\usepackage{amsmath}
\usepackage{amsfonts}
\usepackage{mathtools}
\usepackage{graphicx}
 \usepackage{hyperref}
\usepackage{tikz}
 \mathtoolsset{showonlyrefs,showmanualtags}
  \usepackage{geometry}
\usepackage{cite}
\usepackage{amsthm}
\usepackage[yyyymmdd,hhmmss]{datetime}
\usepackage{stmaryrd}
\newtheorem{theorem}{Theorem}[section]
\newtheorem{lemma}[theorem]{Lemma}
\newtheorem{remark}[theorem]{Remark}

\newtheorem{definition}[theorem]{Definition}
\usepackage{amssymb}
\newcommand{\R}{\mathbb{R}  }

\newcommand{\pfeil}{ \rightarrow }

\newcommand{\into}{\hookrightarrow}

\renewcommand{\phi}{\varphi}
\renewcommand{\div}{\operatorname{div}}
\newcommand{\ljump}{ \llbracket  }
\newcommand{\rjump}{ \rrbracket }

\title[Stability analysis for Mullins-Sekerka with $90^\circ$ angle]{Stability analysis for stationary solutions of the Mullins-Sekerka flow with boundary contact}
\author[Harald Garcke]{Harald Garcke}
\address{Harald Garcke, Fakult\"at f\"ur Mathematik, Universit\"at Regensburg, 93053 Regensburg, Germany}
\author[Maximilian Rauchecker]{Maximilian Rauchecker}
\address{Maximilian Rauchecker, Fakult\"at f\"ur Mathematik, Universit\"at Regensburg, 93053 Regensburg, Germany}

\numberwithin{equation}{section} 

\begin{document}
\maketitle
\begin{abstract}
We first give a complete linearized stability analysis around stationary solutions of the Mullins-Sekerka flow with $90^\circ$ contact angle in two space dimensions. The stationary solutions include flat interfaces, as well as arcs of circles. We investigate the different stability behaviour in dependence of properties of the stationary solution, such as its curvature and length, as well as the curvature of the boundary of the domain at the two contact points. We show that the behaviour changes in terms of these parameters, ranging from exponential stability to instability.

We also give a first result on nonlinear stability for curved boundaries.
\end{abstract}
\section{Introduction}
In this article we study the two-phase Mullins-Sekerka problem inside a bounded, smooth domain in two space dimensions with boundary contact. The sharp interface separating two phases is allowed to meet the boundary of the domain at a constant ninety degree angle. This leads to a contact angle problem for the free boundary.

Let us precisely state the problem.
We consider a fixed, smooth, and bounded domain in two space dimensions $\Omega \subset \R^2$.   
We assume that the domain can be decomposed as $\Omega = \Omega^+ (t) \; \dot\cup \; \mathring \Gamma (t) \; \dot \cup \; \Omega^- (t)$, where $\mathring \Gamma (t)$ denotes the interior of $\Gamma(t)$, a 
smooth one-dimensional curve with two boundary points on $\partial\Omega$. 

The interpretation is that $\Gamma(t)$ is the sharp interface curve separating the two phases, $\Omega^\pm (t)$, which we assume both to be connected. The two boundary points of $\Gamma(t)$ will be denoted by $\partial \Gamma (t)$, the unit normal vector field on $\Gamma(t)$ pointing from $\Omega^-(t)$ to $\Omega^+(t)$ will be denoted by $n_{\Gamma(t)}$. 

The Mullins-Sekerka problem with ninety degree angle condition then reads as
\begin{equation} \label{345huihguzo3gzog345} \tag{1.1}
\begin{alignedat}{2}
V_\Gamma &= - \ljump n_\Gamma \cdot \nabla \mu \rjump, \quad\quad && \text{on } \Gamma(t), \\
\mu|_\Gamma &=  H_\Gamma, &&\text{on } \Gamma(t), \\
\Delta \mu &= 0, &&\text{in } \Omega \backslash \Gamma(t), \\
n_{\partial\Omega} \cdot \nabla \mu|_{\partial \Omega} &= 0, &&\text{on } \partial\Omega \backslash \partial\Gamma(t), \\
n_{\Gamma} \cdot n_{\partial\Omega} &= 0, &&\text{on } \partial \Gamma(t), \\
\Gamma (0) &= \Gamma_0.
\end{alignedat}
\end{equation}
Here, $V_{\Gamma}$ denotes the normal velocity and
$H_\Gamma$ the (mean) curvature of $\Gamma(t)$ with respect to $n_\Gamma$, where we use the sign convention that $H$ is negative for convex spheres. In particular, the sphere of radius $R > 0$ and center $x_0 \in \R^2$ has negative curvature $-1/R$.

The jump brackets $\ljump \cdot \rjump$ are defined by means of 
\begin{equation}
\ljump f \rjump (x) := \lim_{ \varepsilon \pfeil 0+} [ f(x + \varepsilon n_{\Gamma(t)} ) -  f(x - \varepsilon n_{\Gamma(t)} ) ], \quad x \in \Gamma(t).
\end{equation}

Equation $(1.1)_5$ prescribes the ninety degree angle at which the interface $\Gamma(t)$ has contact with the fixed boundary $\partial\Omega$.

We recall, cf. \cite{mulsekpaper12}, that the volume of each of the two phases is conserved, whereas the surface area of the free interface is non-increasing,
\begin{equation} \label{9348765037465087345}
\frac{d}{dt}| \Omega^\pm (t) | = 0, \quad \frac{d}{dt}| \Gamma(t) | \leq 0, \quad t >0.
\end{equation}

Existence of strong solutions to the Mullins-Sekerka problem with $90^\circ$ contact angle (1.1) in an $L_p-L_q$-setting was shown recently in \cite{mulsekpaper12}. The strategy there was to pick some reference surface $\Sigma$ inside the domain, also intersecting the boundary at a ninety degree angle, and write the moving interface as a graph over $\Sigma$ by a height function $h$, depending on space and time. Then pulling back the equations to a time-independent domain, the problem is reduced to a nonlinear evolution equation for $h$.
The height function $h$ in \cite{mulsekpaper12} belongs to
\begin{equation} \label{340958349648569476} \tag{1.2}
h \in W^1_p(0,T;W^{1-1/q}_q(\Sigma)) \cap L_p(0,T;W^{4-1/q}_q(\Sigma)),
\end{equation}
where $p$ and $q$ are different. For further discussion we refer to the article \cite{mulsekpaper12}. We recall that the trace space of $h$ in (1.2) is given by means of real interpolation,
\begin{equation}
X_\gamma := ( W^{4-1/q}_q(\Sigma), W^{1-1/q}_q(\Sigma) )_{1-1/p,p} = B^{4-1/q-3/p}_{qp}(\Sigma),
\end{equation}
and continuously embeds into $C^2(\Sigma)$, cf. Amann \cite{amannlineartheory}. Hence the family of graphs given by $h$ in class (1.2) belongs to $BUC([0,T]; C^2(\Sigma))$.
\subsection*{Preliminaries and function spaces}
Let $d \in \mathbb N$.
The classical $L_p$-Sobolev spaces on $\R^d$ are denoted by $W^k_p(\R^d)$, where $k \in \mathbb N$ and $1 \leq p \leq \infty$. For 
$s \in \R$, the Sobolev-Slobodeckij spaces are denoted by $W^s_p(\R^d)$. By $B^s_{pr}(\R^d)$ we understand the classical Besov spaces on $\R^d$, where $s \in \R, 1 \leq p,r \leq \infty$. We denote by $BUC(\R^d)$ the space of all bounded and uniformly continuous functions.

These function spaces on a bounded, smooth domain $\Omega \subset \R^d$ are as usual defined by restriction. For a given Banach space $X$, the $X$-valued versions of these spaces are denoted by $L_{p}(\Omega;X)$, $W^{k}_{p}(\Omega;X)$, $W^{s}_{p}(\Omega;X)$, $B^{s}_{pr}(\Omega;X)$, and $BUC(\Omega; X)$, respectively. For further discussion we refer to \cite{meyriesveraarpointwise}.

For results on embeddings, traces, interpolation and extension operators we refer to \cite{abelsbuch}, \cite{danchinbuch}, \cite{pruessbuch}, \cite{runst}, \cite{triebel}.

Let us recall the definition of maximal regularity.
\begin{definition}
Let $X$ be a Banach space, $ J = (0,T), 0 < T < \infty$ or $J = \mathbb R_+$. Let $A$ be a closed, densely defined operator on $X$ with domain $D(A) \subset X$. Then the operator $A$ has maximal $L_p$-regularity on $J$, if and only if for every $f \in L_p(J;X)$ there is a unique solution $u \in W^1_p(J;X) \cap L_p(J;D(A))$ to
\begin{equation}
\frac{d}{dt}u (t) + A u (t)= f(t), \quad t \in J, \qquad u|_{t=0} = 0,
\end{equation}
in an almost-everywhere sense in $L_p(J;X)$.
\end{definition}
For a discussion on maximal regularity
we refer to \cite{refpaperpruess}, \cite{prmaxreglpspace}, \cite{dore2000}, \cite{bourgain1984}, \cite{amann}, and \cite{amannlineartheory}.
\subsection*{Outline of paper}
In Section 2 we introduce curvilinear coordinates and transform the problem to a fixed configuration. In Section 3 we derive the full linearization of the Mullins-Sekerka problem at a stationary solution, which may be flat or curved.
Section 4 is devoted to stability and instability results for flat stationary solutions, whereas Section 5 deals with stationary solutions which are arcs of circles. We give an overview of the results on linearized stability in Section 6.
 Section 7 contains a result on nonlinear stability for the case of flat stationary solutions.

\section{Curvilinear coordinates and Hanzawa transform}
In this section we transform the moving free boundary problem (1.1) to a fixed reference configuration, cf. \cite{mulsekpaper12}.

 Let $\Sigma \subset \Omega$ be a smooth reference surface such that $\angle(\Sigma,\partial\Omega) = \pi/2$ on the boundary points $\partial\Sigma \subset \partial\Omega$. Proposition 3.1 in \cite{vogel} states the existence of curvilinear coordinates at least in a small neighbourhood of $\Sigma$. More precisely, there is some possibly small $a = a(\Sigma,\partial\Omega)>0$, such that
\begin{equation}
 X : \Sigma \times (-a,a) \pfeil \R^2, \quad (p,w) \mapsto X(p,w),
\end{equation}
is a smooth diffeomorphism onto its image and $X(.,.)$ is a curvilinear coordinate system, see also \cite{depnerdiss}, \cite{MR2438774}. In particular,
\begin{equation}
 X(p,0) = p, \quad p \in \Sigma,
\end{equation}
hence $D_p  X = I$ on $\Sigma$, as well as
\begin{equation}
D_w  X (p,0) = n_\Sigma (p), \quad p \in \Sigma.
\end{equation}
Furthermore, points on the boundary $\partial\Omega$ only get transported along the boundary, 
\begin{equation}
 X(p,w) \in \partial\Omega, \quad p \in \partial\Sigma, w \in (-a,a). 
\end{equation}
 We need to make use of these coordinates since the boundary $\partial\Omega$ may be curved and therefore a transport only in normal direction is not sufficient. 

We may now parametrize the free interface as follows.
We assume that the free interface is given as a graph over $\Sigma$, that is, there is some height function $h$, such that
\begin{equation}
\Gamma(t) = \Gamma_h (t) := \{ X(p,h(p,t)) : p \in \Sigma \}, 
\end{equation}
for small $t > 0$, at least.
With the help of this coordinate system we may construct a Hanzawa-type transform as follows. The idea goes back to Hanzawa \cite{hanzawa1981}.

Let $\chi \in C_0^\infty(\R)$ be a fixed bump function satisfying $\chi(s)=1$ for $|s| \leq 1/3$, $\chi(s) = 0$ for $|s| \geq 2/3$ and $|\chi'(s)|\leq 4$ for all $s \in \R$. Set $\Sigma_a := X(\Sigma \times (-a,a))$. For a given height function $h$ define
\begin{equation}
F_h (p,w) := \left( p, w - \chi( (w-h(p))/a ) h(p) \right), \quad p \in \Sigma, w \in (-a,a),
\end{equation}
and set
\begin{equation}
\Theta_h (x) := \begin{cases}
(X \circ F_h \circ X^{-1} )(x), & x \in \Sigma_a, \\
x, & x \not\in \Sigma_a. 
 \end{cases}
\end{equation}
The set of admissible height functions is given by 
\begin{equation}
\mathcal U := \{ h \in X_\gamma : |h|_{L_\infty(\Sigma)} < a/5 \},
\end{equation}
where $X_\gamma := B^{4-1/q-3/p}_{qp}(\Sigma)$.
 The following result can be found in \cite{mulsekpaper12}.
\begin{theorem}
Given $h \in \mathcal U$, the transformation $\Theta_h : \Omega \pfeil \Omega$ is a $C^1$-diffeomorphism satisfying $\Theta_h ( \Gamma_h ) = \Sigma$.
\end{theorem}

We can then express the equations in the fixed reference configuration by means of the Hanzawa transform $\Theta_h$, cf. \cite{mulsekpaper12}, \cite{navmulsekpap}, \cite{rtipaper}, \cite{wilkehabil}, \cite{pruessbuch}, \cite{eschersimonett}, \cite{Abelsarabgarcke}.
The transformed system reads as
\begin{equation} \label{fullnonlidfgtrhrthenar8765} \tag{2.1}
\begin{alignedat}{2}
\partial_t h &= - a(h)\ljump n_{\Gamma_h} \cdot \nabla_h \eta \rjump , \quad\quad &&\text{on } \Sigma, \\
\eta|_{\Sigma} &= K(h), && \text{on } \Sigma, \\
\Delta_h \eta &= 0, && \text{in } \Omega \backslash \Sigma, \\
n_{\partial\Omega}^h \cdot \nabla_h \eta|_{\partial\Omega} &= 0, && \text{on } \partial\Omega, \\
n_{\partial\Omega}^h \cdot n_{\Gamma_h} &= 0, && \text{on } \partial\Sigma, \\
h|_{t=0} &= h_0, && \text{on } \Sigma.
\end{alignedat}
\end{equation}
Hereby, $K(h)$ is the transformed (mean) curvature operator, cf. \cite{mulsekpaper12},  $n_{\partial\Omega}^h := n_{\partial\Omega} \circ \Theta_h^t$, and the transformed differential operators are given by
\begin{equation}
\nabla_h := (D\Theta_h^t)^\top \nabla, \quad \div_h  := \operatorname{Tr}\nabla_h , \quad \Delta_h := \div_h \nabla_h.
\end{equation}

Furthermore, $h_0$ is a suitable description of the initial configuration at time $t=0$ and $a(h) (t)$ depends only on $h(t)$ and satisfies $a(0) = 1$.

\section{The linearized problem}
Let $\Sigma_*$ be a stationary solution to the Mullins-Sekerka problem with boundary contact (1.1). In particular, the (mean) curvature of $\Sigma_*$ is constant and $\Sigma_*$ is a flat surface or part of a circle intersecting $\partial\Omega$ perpendicularly. We now consider the full linearization of (1.1) at any stationary solution $\Sigma_*$.
 
Referring to \cite{depnerdiss}, \cite{MR3076297}, given an equilibrium solution $\Sigma_*$, the linearization of the transformed mean curvature operator at $h = 0$ is given by
\begin{equation}
K' (0) = \Delta_{\Sigma_*} + | \kappa_*|^2,
\end{equation}
where $\kappa_*$ is the constant curvature of $\Sigma_*$. Furthermore, the linearization of the nonlinear ninety degree angle condition at the boundary at $h=0$ is given by
\begin{equation}
\nabla_{\Sigma_*} h \cdot n_{\partial\Sigma_*} = -S_{\partial\Omega}(n_{\Sigma_*},n_{\Sigma_*})h, \quad\text{on } \partial \Sigma_*.
\end{equation}
Here $S_{\partial\Omega}(n_{\Sigma_*},n_{\Sigma_*})$ is the second fundamental form of $\partial \Omega$ with respect to the outer unit normal $n_{\partial\Omega}$, cf. \cite{depnerdiss}, \cite{MR3076297}. In particular, we have the formula
\begin{equation}
S_{\partial\Omega}(n_{\Sigma_*},n_{\Sigma_*}) h = - (n_{\Sigma_*} \cdot \partial_{ n_{\Sigma_*} } n_{\partial\Omega} ) h =  - (n_{\Sigma_*} \cdot [D   n_{\partial\Omega} n_{\Sigma_*} ] ) h,
\end{equation}
cf. the proof of Lemma 3.7 in \cite{depnerdiss}.
Note that if e.g. $\Omega$ is a convex sphere, $S_{\partial\Omega} < 0$. The linearized problem around a stationary solution $\Sigma_*$ now reads as
 \begin{equation} \label{345huifghfghfgg34g5} \tag{3.1}
\begin{alignedat}{2}
\partial_t h &= - \ljump n_{\Sigma_*} \cdot \nabla \mu \rjump, \quad\quad\quad && \text{on } \Sigma_*, \\
\mu|_{\Sigma_*} &= \Delta_{\Sigma_*} h + \kappa_*^2 h , \quad\quad\quad\quad\quad &&\text{on } \Sigma_*, \\
\Delta \mu &= 0, &&\text{in } \Omega \backslash \Sigma_*, \\
n_{\partial\Omega} \cdot \nabla \mu|_{\partial \Omega} &= 0, &&\text{on } \partial\Omega \backslash \partial\Sigma_*, \\
\nabla_{\Sigma_*} h \cdot n_{\partial\Sigma_*} &= -S_{\partial\Omega}(n_{\Sigma_*},n_{\Sigma_*})h, &&\text{on } \partial \Sigma_*, \\
h (0) &= h_0, &&\text{on }\Sigma_* .
\end{alignedat}
\end{equation}

Regarding the stationary solution we note that $\kappa_*$ is constant and either equal to zero or $-1/R$, because $\Sigma_*$ is flat or part of a circle with radius $ R>0$, respectively.

To identify relevant quantities in the stability analysis, let us formally consider the corresponding eigenvalue problem
\begin{equation} \label{345fghfghhfgh4g5} \tag{3.2}
\begin{alignedat}{2}
\lambda h &= - \ljump n_{\Sigma_*} \cdot \nabla \mu \rjump, \quad\quad & &\text{on } \Sigma_*, \\
\mu|_{\Sigma_*} &= \Delta_{\Sigma_*} h + \kappa_*^2 h , \quad\quad\quad &&\text{on } \Sigma_*, \\
\Delta \mu &= 0, &&\text{in } \Omega \backslash \Sigma_*, \\
n_{\partial\Omega} \cdot \nabla \mu|_{\partial \Omega} &= 0, &&\text{on } \partial\Omega \backslash \partial\Sigma_*, \\
\nabla_{\Sigma_*} h \cdot n_{\partial\Sigma_*} &= -S_{\partial\Omega}(n_{\Sigma_*},n_{\Sigma_*})h, \quad \quad &&\text{on } \partial \Sigma_*,
\end{alignedat}
\end{equation}
for some $\lambda \in \mathbb C$. Multiplying $(3.2)_1$ with $ \Delta_{\Sigma_*} \bar h + \kappa_*^2 \bar h$ in $L_2(\Sigma_*)$ gives
\begin{equation}
\lambda \int_{\Sigma_*} h ( \Delta_{\Sigma_*} \bar h + \kappa_*^2 \bar h ) d\mathcal H^1 = \int_\Omega | \nabla \mu |^2 dx.
\end{equation}
Here, $d \mathcal H^d$ denotes the $d$-dimensional Hausdorff measure, $d \in \mathbb N_0$.
An integration by parts invoking the boundary conditions entails
 \begin{equation} \label{3984750934785034875034} \tag{3.3}
 \begin{alignedat}{1}
\lambda \left[ \int_{\Sigma_*} | \nabla_{\Sigma_*} h |^2 d\mathcal H^1 + \int_{\partial\Sigma_*} S_{\partial\Omega}(n_{\Sigma_*},n_{\Sigma_*}) |h|^2 d\mathcal H^0 - \kappa_*^2 \int_{\Sigma_*} |h|^2 d\mathcal H^1  \right] + \\ + \int_\Omega | \nabla \mu |^2 dx = 0.
\end{alignedat}
\end{equation}
In particular, (3.3) implies that $\lambda$ is necessarily real.

We now note that the term in brackets may changes its sign in dependence of the curvature $\kappa_*$, the values of the form $S_{\partial\Omega}(n_{\Sigma_*},n_{\Sigma_*})$ on the two boundary points of $\partial\Sigma_*$, and the length of the curve $\Sigma_*$. The last dependence is somewhat hidden and stems from the scaling properties of the first term involving the gradient of $h$. 

Referring to \cite{MR2438774}, we want to introduce a bilinear functional by
\begin{equation}
I_* (h,h) := \int_{\Sigma_*} | \nabla_{\Sigma_*} h |^2 d\mathcal H^1 + \int_{\partial\Sigma_*} S_{\partial\Omega}(n_{\Sigma_*},n_{\Sigma_*}) |h|^2 d\mathcal H^0 - \kappa_*^2 \int_{\Sigma_*} |h|^2 d\mathcal H^1.
\end{equation}
Note that for $\lambda \not = 0$, integrating $(3.2)_1$ over $\Sigma_*$ yields that necessarily $\int_{\Sigma_*} h d \mathcal H^1 = 0$, for any eigenfunction $h$ to the eigenvalue $\lambda \not= 0$.

Hence it stems from (3.3) that positivity of $I_*$ on mean value free functions gives $\lambda \leq 0$ for any possible eigenvalue $\lambda$. Hence studying the sign of $I_*$ for mean value free functions is the crucial point in our stability analysis. We want to remark that $I_*$ is the second derivative of the length functional $[ \Gamma \mapsto \int_{\Gamma} d \mathcal H^1 ]$ for variations keeping the areas of the phases conserved, at the point $\Sigma_*$, cf. Proposition 3.3 in \cite{MR1906593} for a related computation. If now $I_*$ is positive on functions having mean zero, the surface $\Sigma_*$ is a minimum point in the energy landscape shaped by the length functional, hence we expect stability of $\Sigma_*$.
Note that since the volumes of the phases are conserved in time, the set of admissible variations of $\Gamma$ naturally corresponds to mean value free functions $h$.

 If on the other hand $I_*$ is not positive anymore, the stationary point $\Sigma_*$ is no longer a minimum, hence we expect instability of $\Sigma_*$. 

As a trivial consequence to (3.3) we want to point out that if $\kappa_* = 0$ and $S$ is identically zero on $\partial \Sigma_*$, we obtain that $\lambda \leq 0$. This corresponds to the geometrical situation of a flat solution $\Sigma_*$ and flat, perpendicular walls, which was already investigated in \cite{mulsekpaper12}.
\section{Flat stationary solutions}
Let us start with the simpler case when the stationary solution is flat, $\kappa_* = 0$. Then by rotation, we can assume that $\Sigma_* = (0,L)$ for some $L > 0$. Let us rewrite (3.1) as an abstract evolution equation in the setting of \cite{mulsekpaper12}. Let $3/2 < q < 2$ and
\begin{equation}
X_0 := W^{1-1/q}_q(\Sigma_*), \quad X_1 := W^{4-1/q}_q(\Sigma_*).
\end{equation}
Define a linear operator $A : D(A) \subset X_1 \pfeil X_0$ as follows. Let $ B u := \ljump n_{\Sigma_*} \cdot \nabla u \rjump$ and $T_0 v$ be the unique solution of the two-phase elliptic problem
\begin{equation} \label{345fdfgdfgh4g5} \tag{4.1}
\begin{alignedat}{2}
\ljump \mu \rjump = 0, \quad \mu|_{\Sigma_*} &= v , \quad\quad\quad &\text{on } \Sigma_*, \\
\Delta \mu &= 0, &\text{in } \Omega \backslash \Sigma_*, \\
n_{\partial\Omega} \cdot \nabla \mu|_{\partial \Omega} &= 0, &\text{on } \partial\Omega \backslash \partial\Sigma_*.
\end{alignedat}
\end{equation}
Then we define $A$ by $A h := B T_0 ( \Delta_{\Sigma_*} h )$, with domain
\begin{equation}
D(A) := X_1 \cap \{ h : \nabla_{\Sigma_*} h \cdot n_{\partial\Sigma_*} = -S_{\partial\Omega}(n_{\Sigma_*},n_{\Sigma_*})h, \;\text{on } \partial \Sigma_* \}.
\end{equation}
We can then rewrite (3.1) as
\begin{equation} \label{kjlh43kjhj345345435} \tag{4.2}
\dot h + Ah = 0, \; t > 0, \quad h(0) = h_0.
\end{equation}
The main benefit of this formulation is now the fact that $A$ has maximal regularity. More precisely, let $p \in (6,\infty)$, $q \in (19/10,2) \cap (2p/(p+1), 2)$, and $J = (0,T)$, $0 < T < \infty$. Then, by a perturbation argument, the operator $A$ has maximal $L_p$-regularity on $J$ with respect to the base space $X_0$, cf. Theorem 4.8 in \cite{mulsekpaper12}. Define the trace space as $X_\gamma := B^{4-1/q-3/p}_{qp}(\Sigma_*)$. Then it holds that
\begin{equation}
X_\gamma = (X_0, X_1)_{1-1/p,p}.
\end{equation}
We now want to apply the generalized principle of linearized stability to deduce stability or instability results for (4.2), cf. \cite{pruessbuch}, \cite{pruessmulsek}.

Let us simplify notation first. Since $\partial \Sigma_* = \{0,L\}$, we may rewrite the boundary conditions as
\begin{equation}
\partial_x h (0) = - \omega_1 h(0), \quad \partial_x h (L) = \omega_2 h(L),
\end{equation}
where
\begin{equation}
\omega_1 := -S_{\partial\Omega}(n_{\Sigma_*},n_{\Sigma_*}) (0), \quad \omega_2 := -S_{\partial\Omega}(n_{\Sigma_*},n_{\Sigma_*}) (L).
\end{equation}
In particular again if $\Omega$ is a convex sphere, we have $\omega_1, \omega_2 > 0$ since $S_{\partial \Omega} < 0$. We now want to analyse different geometries and their respective stability properties.
\begin{figure}[h]
\begin{tikzpicture}
\draw (-1.6,0) arc (270:90:0.5cm);
\draw (-4.1,0) arc (-90:90:0.5cm);
\draw (-3.6,0.5) -- (-2.1,0.5);


\draw (4.5,0) arc (270:90:0.5cm);
\draw (5.5,0) arc (-90:90:0.5cm);
\draw (4,0.5) -- (6,0.5) ;

\draw (0.15,0) -- (0.15,1);
\draw (2.15,0) -- (2.15,1);
\draw (0.15,0.5) -- (2.15,0.5);

\draw (5.1,0.25) node {$\Sigma_\ast$};
\draw (1.3,0.25) node {$\Sigma_\ast$};
\draw (-2.85,0.25) node {$\Sigma_\ast$};
\draw (2.5,0.4) node {$\partial\Omega$};
\draw (-1.7,0.3) node {$\partial\Omega$};
\draw (6.2,0.2) node {$\partial\Omega$};
\end{tikzpicture}

\caption{Different signs of $\omega_1,\omega_2$. Left: $\omega_1 = \omega_2 <0$. Middle: $\omega_1 = \omega_2 = 0$. Right: $\omega_1= \omega_2 > 0$.}
\label{figuregg}
\end{figure}
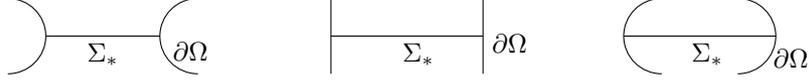
\subsection{Stability and instability results} \label{34095645694586945869456}
Let us start with the left hand side case, where we can show exponential stability of $h_*=0$ for (4.2).
\begin{theorem} \label{jkh34jkh5g34h5345}
Let $\omega_1, \omega_2 \leq 0$. Then $h_* = 0$ is normally stable, that is,
\begin{enumerate}
\item The set of equilibria of \normalfont{(4.2)} is the kernel of $A$, which is one-dimensional.
\item The eigenvalue zero is semi-simple, $X_0 = N(A) \oplus R(A)$.
\item The spectrum satisfies $\sigma (-A) \backslash \{ 0 \} \subset \mathbb R_- := \{ z \in \mathbb C : \operatorname{Re} z < 0, \operatorname{Im} z = 0 \} $.
\end{enumerate}
In particular, $h_* = 0$ is stable in $X_\gamma$ and there is some $\delta > 0$, such that if $|h_0 |_{X_\gamma} \leq \delta$, the unique solution $h$ of \normalfont{(4.2)} exists globally in time, $$h \in W^1_p(\R_+ ; X_0) \cap L_p(\R_+ ; D(A)),$$ and converges to some equilibrium solution in $X_\gamma$ at an exponential rate.
\end{theorem}
\begin{proof}
Consider some $\lambda \in \sigma(-A) \subset \mathbb C$ and the corresponding eigenvalue problem for the eigenfunction $h \in D(A)$,
\begin{equation} \label{345fghfghfgh74g5} \tag{4.3} \begin{cases} 
\begin{alignedat}{2}
\lambda h &= - \ljump n_{\Sigma_*} \cdot \nabla \mu \rjump, \quad\quad & \text{on } \Sigma_*, \\
\ljump \mu \rjump = 0, \quad \mu|_{\Sigma_*} &= \partial_x \partial_x h  , \quad\quad\quad &\text{on } \Sigma_*, \\
\Delta \mu &= 0, &\text{in } \Omega \backslash \Sigma_*, \\
n_{\partial\Omega} \cdot \nabla \mu|_{\partial \Omega} &= 0, &\text{on } \partial\Omega \backslash \partial\Sigma_*, \\
\partial_x h (0) &= - \omega_1 h(0), \\ \quad \partial_x h (L) &= \omega_2 h(L).
\end{alignedat} \end{cases}
\end{equation}
Testing equation $(4.3)_1$ with $\partial_x \partial_x \bar h$, an integration by parts invoking the boundary conditions yields
\begin{equation} \label{324jkl5h3kj4h5l3ghj45} \tag{4.4}
\lambda \left[ \int_0^L | \partial_x h |^2 dx - \omega_1 |h(0)|^2 - \omega_2 |h(L)|^2 \right] + \int_\Omega | \nabla \mu |^2  = 0.
\end{equation}
Let us characterise the kernel of $A$. Pick some $h \in N(A)$. Then (4.4) for $\lambda = 0$ entails that $\mu$ has to be constant, whence $\partial_x \partial_x h = c$ on $(0,L)$ for some $c \in \R$. In particular, by the fundamental theorem of calculus,
\begin{equation}
h(s) = h(0) + s \partial_x h(0) + \int_0^s \int_0^\tau \partial_x \partial_x h (\tau') d\tau' d\tau, \quad s \in [0,L],
\end{equation}
whence invoking the boundary condition gives
\begin{equation} \label{34khju456456jh45} \tag{4.5}
h(s) = h(0) - \omega_1 h(0) s + cs^2/2, \quad s \in [0,L].
\end{equation}
Let us start now with the case where $\omega_1, \omega_2 < 0$. By differentiating (4.5) and invoking the boundary condition at $x=L$ we obtain that
\begin{equation}
\partial_x h (L) = - \omega_1 h(0) + cL = \omega_2 h(L).
\end{equation}
Also from (4.5) we obtain that $h(L) = [1-\omega_1 L] h(0) + cL^2/2$. The linear system
\begin{equation} \label{3456hjkj45h6456} \tag{4.6}
\begin{bmatrix} h(0)  \\ h(L) \end{bmatrix} = \begin{bmatrix} 0 & -\omega_2/\omega_1 \\ 1- \omega_1 L & 0 \end{bmatrix} \begin{bmatrix} h(0)  \\ h(L) \end{bmatrix} + \begin{bmatrix} cL/\omega_1 \\ cL^2/2\end{bmatrix}
\end{equation}
now has a unique solution since $1+(1-\omega_1 L)\omega_2/\omega_1 \geq 1$ for any $\omega_1 <0, \omega_2 <0, L > 0$. Explicitly solving the linear system gives
\begin{equation}
h(0) = \frac{ c(L - \omega_2 L^2 /2) }{ \omega_1 + \omega_2 - \omega_1 \omega_2 L },
\end{equation}
which gives in combination with (4.5) a unique solution $h$ which depends linearly on $c$. Hence the kernel of $A$ is truly one-dimensional.

With this at hand we may now prove that zero is a semi-simple eigenvalue. Since $D(A)$ compactly embeds into $X_0$, the resolvent of $A_0$ is compact on the resolvent set. Therefore the spectrum only consists of at most countably many isolated eigenvalues. Furthermore, every spectral value in $\sigma(A)$ is a pole of finite algebraic multiplicity. Using Remark A.2.4 in \cite{lunardioptimal} it suffices to show that $N(A) = N(A^2)$. Then the range of $A$ is closed and $X_0 = N(A) \oplus R(A)$. So pick some $h \in N(A^2)$. Then $h_1 := Ah \in R(A) \cap N(A)$. Then $h_1$ is mean value free on $(0,L)$ and there is some $c_1 \in \R$ such that
\begin{equation}
h_1 (x) = c_1 \left[ \frac{ (L- \omega_2 L^2/2) (1- \omega_1 x) }{\omega_1 + \omega_2 - \omega_1 \omega_2 L} + \frac{x^2}{2} \right], \quad x \in [0,L].
\end{equation}
A straightforward integration gives
\begin{equation}
\int_0^L h_1 (x)dx = c_1 \frac{6L^2 + \omega_1 \omega_2 L^4 /2 - 2 (\omega_1 + \omega_2) L^3}{6 ( \omega_1 + \omega_2 - \omega_1 \omega_2 L)}.
\end{equation}
Now for any $L >0, \omega_1 , \omega_2 < 0$ the right hand side can only be zero if $c_1 = 0$. But then $h_1 = 0$ and $Ah = 0$. Hence $h$ belongs to the kernel of $A$. Then $N(A) = N(A^2)$ and zero is semi-simple. Furthermore, equation (4.4) yields that necessarily $\lambda$ is real and $\lambda \leq 0$. Hence the third assertion is proved. The rest of the statement is a consequence of the generalized principle of linearized stability of Pr\"uss, Simonett, and Zacher \cite{pruessmulsek}.

For completeness we shall show that zero is also semi-simple in the simpler case where $\omega_ 1 = 0, \omega_2 < 0$. The case $\omega_1 <0 , \omega_2 = 0$, follows the same lines. In the case $\omega_ 1 = 0, \omega_2 < 0$, the kernel of $A$ consists of functions $h$ of form $h_c (s) = c( L/\omega_2 - L^2/2 + s^2 /2)$ for $c \in \R$. Then the same arguments give that zero is semi-simple.
Note that in the case $\omega_1 = \omega_2 = 0$, the kernel of $A$ consists of the constant functions, cf. \cite{mulsekpaper12}.
\end{proof}
Let us now be concerned with the case when $\omega_1, \omega_2 > 0$, cf. Figure 1. We will show the following result for (4.2). For simplicity we will assume that $\omega_1 = \omega_2 =: \omega_+ > 0$.
\begin{theorem} \label{34985jhjjj34h54j5645}
\begin{enumerate}
\item For fixed $L > 0$, there exists some $\delta = \delta(L) > 0$, such that if $\omega_+  \leq \delta$, the solution $h_*=0$ is stable in $X_\gamma$. Furthermore, there exists some $\eta > 0$, such that if $|h_0|_\gamma \leq \eta$, the solution to the initial value $h_0$ exists on $\R_+$ and converges to the equilibrium point $h_\infty := \frac{1}{L} \int_0^L h_0 dx$ in $X_\gamma$ at an exponential rate.
\item For fixed $L > 0$, there exists some $K = K(L) > 0$, such that if $\omega_+  \geq K$, the solution $h_*=0$ is normally hyperbolic and unstable in $X_\gamma$.
\item For fixed $\omega_+ >0$, there is some $\delta > 0$ such that if $L_0 \leq \delta$, the interface $\Sigma = (0,L_0)$ corresponding to $h_* = 0$ is stable in $X_\gamma$. Moreover, the second statement of (1) holds.
\item For fixed $\omega_+ > 0$, there is some $K > 0$, such that if $L \geq K$, the interface $\Sigma_* = (0,L)$ corresponding to $h_* = 0$ is normally hyperbolic and unstable in $X_\gamma$.
\end{enumerate}
\end{theorem}
\begin{proof}
Let $L , \omega_+ > 0$ and $\Sigma_* = (0,L)$. Let $A$ be the linear operator of (4.2). Let us be concerned with the kernel of $A$. Again if $Ah = 0$, the corresponding chemical potential $\mu = T_0 \partial_x \partial_x h$ is constant and therefore $\partial_x \partial_x h = c$ for some $c \in \R$. As before, $h$ can be written as $h(s) = (1 - \omega_+ s)h(0) + cs^2/2$ for all $s \in [0,L]$. The corresponding linear system for $[h(0),h(L)] \in \R^2$ in (4.6) can be uniquely solved whenever $2 - L\omega_+ \not = 0$. Note that for either fixed $L > 0$ or $\omega_+ > 0$, this can be ensured by choosing $\delta > 0$ sufficiently small or $K > 0$ sufficiently large. In any case,
\begin{equation} \label{j3kh45k345} \tag{4.7}
h(s) = h(0)[1-\omega_+ s] + cs^2/2, \; s \in [0,L], \quad h(0) = c \frac{ L - \omega_+ L^2/2}{\omega_+ (2-\omega_+ L)}.
\end{equation}
Arguing as in the proof of Theorem 4.1 we can show that the kernel of $A$ is truly one dimensional, given by functions of type (4.7) for $c \in \R$. Hence, 
$X_0 = N(A) \oplus R(A)$ and the eigenvalue zero is semi-simple.

(1) For some $0 \not= \lambda \in \sigma(-A) \subset \mathbb C$ and a corresponding eigenfunction $h \in D(A)$, the eigenvalue problem again reads as
\begin{equation} \label{345dfgdfgdfg5} \tag{4.8}
\begin{alignedat}{2}
\lambda h &= - \ljump n_{\Sigma_*} \cdot \nabla \mu \rjump, \quad\quad & \text{on } \Sigma_*, \\
\ljump \mu \rjump = 0, \quad \mu|_{\Sigma_*} &= \partial_x \partial_x h  , \quad\quad\quad &\text{on } \Sigma_*, \\
\Delta \mu &= 0, &\text{in } \Omega \backslash \Sigma_*, \\
n_{\partial\Omega} \cdot \nabla \mu|_{\partial \Omega} &= 0, &\text{on } \partial\Omega \backslash \partial\Sigma_*, \\
\partial_x h (0) &= - \omega_+ h(0), \\ \partial_x h (L) &= \omega_+ h(L).
\end{alignedat}
\end{equation}
Necessarily,
\begin{equation}
\lambda \left[ \int_0^L | \partial_x h |^2 dx - \omega_+ |h(0)|^2 - \omega_+ |h(L)|^2 \right] + \int_\Omega | \nabla \mu |^2  = 0.
\end{equation}
We aim to show that the term in brackets is still positive if $\omega_+ \leq \delta$ and $\delta > 0$ is small. Integrating $\eqref{345dfgdfgdfg5}_1$ over $(0,L)$ yields that $h$ is mean value free. Hence we can use the Poincare-Wirtinger inequality to deduce
\begin{equation}\label{3498540958760456} \tag{4.9}
\begin{alignedat}{1} 
\int_0^L | & \partial_x h |^2 dx - \omega_+ |h(0)|^2 - \omega_+ |h(L)|^2 \\ &\geq c_0(L) |h|_{H^1_2(0,L)}^2 - \omega_+ |h(0)|^2 - \omega_+ |h(L)|^2 \\
&\geq \tilde c_0(L) |h|_{C^0([0,L])}^2 - \omega_+| h(0)|^2 - \omega_+ |h(L)|^2,
\end{alignedat}
\end{equation}
for some $c_0, \tilde c_0 > 0$, since $H^1_2(0,L) \into C^0 ([0,L])$. Hence the first claim follows if $\delta > 0$ is sufficiently small.

(2) Again we fix $L > 0$. We need to show that if $\omega_+ \geq K$ for $K >0$ large, there is a positive eigenvalue $\lambda > 0$ of $-A$. For $\lambda > 0$ we can rewrite the eigenvalue problem
(4.8) as
\begin{equation} \label{3ghjgdfghj957} \tag{4.10}
\lambda h - D_{MS} \tilde \Delta h = 0,
\end{equation}
where $\tilde \Delta : D(\tilde \Delta) \subset X_1 \pfeil X_0$ is given by $\tilde \Delta h := \partial_x \partial_x h$ with domain $D(\tilde \Delta) := W^{4-1/q}_q (\Sigma_*) \cap \{ \partial_x h (0) = - \omega_+ h(0), \; \partial_x h (L) = \omega_+ h(L) \}$. Furthermore, we define the Dirichlet-to-Neumann operator $D_{MS}$ as follows. For given $g \in W^{2-1/q}_q(0,L)$, we solve the two-phase elliptic problem
\begin{align}
\Delta \theta &= 0, &\text{in } \Omega \backslash \Sigma_*, \\
\ljump \theta \rjump = 0, \quad \theta|_{\Sigma_*} &= g  , \quad\quad\quad &\text{on } \Sigma_*, \\
n_{\partial\Omega} \cdot \nabla \theta|_{\partial \Omega} &= 0, &\text{on } \partial\Omega \backslash \partial\Sigma_*, 
\end{align}
uniquely by $\theta \in W^2_q(\Omega \backslash \Sigma_*)$ and define $D_{MS} g := - \ljump n_{\Sigma_*} \cdot \nabla \theta \rjump$. The inverse Neumann-to-Dirichlet operator $$N_{MS} = [D_{MS}]^{-1} : W^{1-1/q}_{q,(0)}(\Sigma_*) \pfeil W^{2-1/q}_{q,(0)}(\Sigma_*)$$ then admits a compact, selfadjoint extension to $L_{2,(0)}(\Sigma_*)$, cf. \cite{rtipaper}. It is also shown there that
$N_{MS}$ is injective on $L_{2,(0)}(\Sigma_*)$. Note however that $\partial_x \partial_x h$ is not mean value free on $\Sigma_*$, even though $h$ is. We may however rewrite (4.10) as
\begin{equation} \label{32490857348957}
\lambda h - D_{MS} (I-P_0)\tilde \Delta h = -D_{MS} P_0 \tilde \Delta h,
\end{equation}
where $P_0 v$ is the mean value of $v$. Next note that $D_{MS} P_0 \tilde \Delta h = 0$, since $P_0 \tilde \Delta h$ is constant. Applying $N_{MS}$ then gives that (4.10) is equivalent to
\begin{equation} \label{32490857348957}
\lambda N_{MS} h -  (I-P_0)\tilde \Delta h = 0.
\end{equation}
Hereby we understand $\tilde \Delta$ as the natural extension to $H^2_2(\Sigma_*)$. We may now follow the lines of \cite{rtipaper}, \cite{wilkehabil}, \cite{pruessbuch}. Define $B_\lambda := \lambda N_{MS} - (I- P_0)\tilde \Delta$ with natural domain $D(B_\lambda) := H^2_{2,(0)}(\Sigma_*) \cap \{ \partial_x h (0) = - \omega_+ h(0), \; \partial_x h (L) = \omega_+ h(L) \}$.
We will now show that
\begin{equation} \label{34534534853489ghjghjghj5345} \tag{4.11}
B_\lambda \text{ is } \begin{cases}
\text{positive definite,} &\text{if $\lambda \geq \lambda_0$ for some $\lambda_0 > 0$}, \\
\text{not positive definite,} &\text{if $\lambda > 0$ is sufficiently small}.
\end{cases}
\end{equation}
In \cite{rtipaper} it is shown that $N_{MS}$ is positive definite on $L_{2,(0)}(\Sigma)$, hence there is some $d_0 > 0$ such that
\begin{align}
(B_\lambda h |h)_2 &= \lambda (N_{MS} h|h)_2 - (\partial_x \partial_x h | h)_2 + (P_0 \partial_x \partial_x h|h)_2 \\
&\geq \lambda d_0 |h|_2^2 + | \partial_x h|_2^2 - \omega_+ [ h(0)^2 + h(L)^2],
\end{align}
since $(P_0 \partial_x \partial_x h|h)_2 = 0$. It remains to show that 
\begin{equation} \label{9867087560785sdf} \tag{4.12}
(\lambda -1)d_0 |h|_2^2 + | \partial_x h|_2^2 - \omega_+ [ h(0)^2 + h(L)^2] \geq 0,
\end{equation}
if only $\lambda \geq \lambda_0$ for $\lambda_0 > 0$ sufficiently large. Then $B_\lambda$ is positive definite for $\lambda \geq \lambda_0$. We now claim the following Young-type inequality. Note that the following lemma immediately implies (4.12).
\begin{lemma} \label{2345hijou34h5u3h45ij34}
For every $\delta > 0$ there is a constant $C_\delta > 0$, such that
\begin{equation}
h(j)^2 \leq \delta | \partial_x h|^2_2 + C_\delta |h|_2^2, \quad j = 0,L,
\end{equation}
for any $h \in H^1_{2}(\Sigma_*)$.
\end{lemma}
\begin{proof}
The proof follows the lines of \cite{garckeitokosaka}. Assume there is some $\delta > 0$ such that the statement is not true. Then there is a sequence $(h_n)_n \subset H^1_{2}(\Sigma_*)$, such that 
\begin{equation} \tag{4.13}
\label{3849573408753} 1 = h_n(0)^2 > \delta | \partial_x h_n|_2^2 + n |h_n|_2^2, \quad \text{for all } n \in \mathbb N.
\end{equation}
In particular, $| \partial_x h_n|_2^2 < 1/\delta$ and $|h_n|_2^2 < 1/n$ for each $n$. Hence $(h_n)_n$ is bounded in $H^1_2$ and there is a subsequence again denoted by $(h_n)_n$ converging weakly to some $h$ in $H^1_{2}$. Furthermore, $h_n$ converges strongly to zero in $L_2$. By uniqueness, $h_n$ converges weakly to zero in $H^1_2$. By the compact embedding $H^1_2(\Sigma_*) \into \into C^0( [0,L] )$, $h_n$ converges strongly in $C^0$-norm to zero as $n \pfeil \infty$. This implies $h_n (0) \pfeil 0$ as $n \pfeil \infty$, which is a contradiction to (4.13). 
\end{proof}
We now show the second part of (4.11). Note that
\begin{equation}
\lim_{\lambda \pfeil 0, \lambda > 0} (B_\lambda h |h)_2 = - (\partial_x \partial_x h | h)_2 =  | \partial_x h|_2^2 - \omega_+ [ h(0)^2 + h(L)^2], \quad h \in D(B_\lambda),
\end{equation}
since $\lim_{\lambda \pfeil 0, \lambda > 0}  \lambda (N_{MS} h|h)_2 =0$ for any fixed $h \in D(B_\lambda)$. It now remains to construct a function $\bar h \in D(B_\lambda)$ such that
\begin{equation} \label{23489dfg53478345} \tag{4.14}
| \partial_x \bar h|_2^2 - \omega_+ [ \bar h(0)^2 + \bar h(L)^2] < 0.
\end{equation}
We start with the following construction. Let $\varepsilon > 0$ and define 
\begin{equation} \label{3453fghfgh345} \tag{4.15}
\bar g(s) := \begin{cases}
1- \omega_+ s, & s \in [0, \varepsilon], \\
1- \omega_+ \varepsilon - (1-\omega_+\varepsilon)(s-\varepsilon)/(L/2-\varepsilon), & s \in [\varepsilon, L/2], \\
-g(s-L), & s \in [L/2, L],
\end{cases}
\end{equation}
cf. Figure 2.
\begin{figure}[h]
\begin{tikzpicture}[scale = 0.65]
\draw[step=0.5cm,gray,very thin] (-1.9,-2.9) grid (9.9,2.9);
\draw[thick,->] (0,0) -- (9,0) node[anchor=north west] {$x$};
\draw[thick,->] (0,-2.5) -- (0,2.5) node[anchor=south east] {$\bar g (x)$};
    \draw [thick] (0.5 cm,2pt) -- (0.5 cm,-2pt) node[anchor=north] {$\varepsilon$};
    \draw[thick]  (4 cm,2pt) -- (4 cm,-2pt) node[anchor=north] {$L/2$};
        \draw[thick]  (7.5 cm,2pt) -- (7.5 cm,-2pt) node[anchor=south] {$L - \varepsilon$};
          \draw[thick]  (8 cm,2pt) -- (8 cm,-2pt) node[anchor=north] {$L$};
\foreach \y in {-1,0,1}
    \draw [thick] (1pt,-2 cm) -- (-1pt,-2 cm) node[anchor=east] {$-1$};
    \draw[thick]  (1pt,2 cm) -- (-1pt,2 cm) node[anchor=east] {$1$};
      \draw [thick] (0,2) -- (0.5,1);
      \draw [thick] (0.5,1) -- (7.5,-1);
         \draw [thick] (8,-2) -- (7.5,-1);
\end{tikzpicture}
\caption{Construction of $\bar g$.}
\label{figurdfgdfgdfg}
\end{figure}
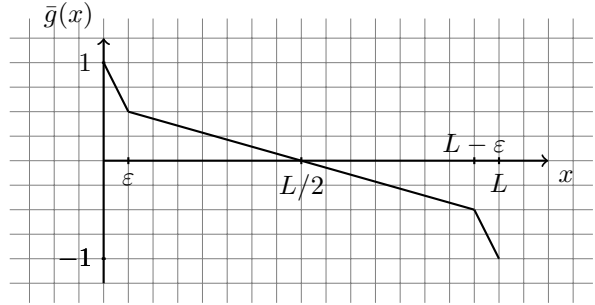
Then $\bar g$ satisfies the boundary conditions $\partial_x \bar g(0) = - \omega_+ \bar g(0)$ and  $\partial_x \bar g(L) = \omega_+ \bar g(L)$. Clearly, $g \in H^1_{2,(0)} (\Sigma_*)$. Furthermore a direct calculation shows
\begin{align} 
| \partial_x \bar g|_{L_2(0,L/2)}^2 - \omega_+ \bar g(0)^2 &= \int_0^\varepsilon \omega_+^2 dx + \int_\varepsilon^{L/2} \frac{(1 - \omega_+ \varepsilon)^2}{ (L/2 - \varepsilon)^2} dx - \omega_+ \\
&= \varepsilon \omega_+^2  +  \frac{(1 - \omega_+ \varepsilon)^2}{ L/2 - \varepsilon}  - \omega_+.
\end{align}
In particular, the first two terms converge to $2/L$ as $\varepsilon \pfeil 0$. If now $\omega_+ > 2/L$, the right hand side will be negative whenever $\varepsilon = \varepsilon(\omega_+) >0$ is small enough. Note that the critical value for $\omega_+$ is $2/L$, which is exactly the degeneracy of the linear system for $[h(0),h(L)] \in \R^2$ in (4.6): $2 - L\omega_+  = 0$. By approximating $\bar g$ with a smoother function in $D(B_\lambda)$ we have shown (4.14). Following \cite{wilkehabil} using (4.11) we obtain that there is indeed a positive eigenvalue $\lambda > 0$ as claimed.

(3) Fix now $\omega_+ > 0$. We now need to understand the dependence on $L$ in estimate (4.9). 
Let us calculate the embedding constant of $H^1_{2,(0)} (0,L) \into C^0 ([0,L])$. Firstly,
\begin{equation}
h(t) = h(s) + \int_s^t \partial_x h (\tau) d\tau, \quad s,t \in [0,L].
\end{equation}
Integrating over $s \in [0,L]$ and using that $h$ is mean value free on $(0,L)$ gives
\begin{equation}
h(t) =  \frac{1}{L} \int_0^L \int_s^t \partial_x h (\tau) d\tau ds , \quad s, t \in [0,L].
\end{equation}
Hence
\begin{equation}
\sup_{t \in [0,L]} |h(t)| \leq  \int_0^L | \partial_x h (\tau) | d\tau.
\end{equation}
H\"olders inequality gives
\begin{equation}
\sup_{t \in [0,L]} |h(t)|^2 \leq  L \int_0^L | \partial_x h (\tau) |^2 d\tau.
\end{equation}
In particular,
\begin{equation}\label{34985dfgdfg0456}
\begin{alignedat}{1} 
\int_0^L | & \partial_x h |^2 dx - \omega_+ h(0)^2 - \omega_+ h(L)^2 \\ &\geq \frac{1}{L} |h|_{C^0([0,L])}^2 - \omega_+ [ h(0)^2 + h(L)^2 ].
\end{alignedat}
\end{equation}
Again we see that if $L = L(\omega_+) >0$ is sufficiently small, 
\begin{equation}\label{34985dfgdfg0456}
\begin{alignedat}{1} 
\int_0^L | & \partial_x h |^2 dx - \omega_+ [h(0)^2 + h(L)^2 ]  \geq (\frac{1}{L} - 2\omega_+) |h|_{C^0([0,L])}^2 \geq 0.
\end{alignedat}
\end{equation}
Hence (3) follows.

(4) We fix $\omega_+ > 0$. Following the lines of the proof of (2), we only need to justify $(4.11)_2$, where $B_\lambda$ is defined as before. In particular,
we need to show that there is a function $\bar h \in D(B_\lambda)$ such that
\begin{equation} \label{23489dfg53hjvjh478345}
| \partial_x \bar h|_{L_2(0,L)}^2 - \omega_+ [ \bar h(0)^2 + \bar h(L)^2] < 0,
\end{equation}
where now $\omega_+ >0$ is fixed, if we only choose $L>0$ large enough. Let us consider the function $\bar g$ defined in (4.15). Again for $\varepsilon > 0$,
\begin{align} 
| \partial_x \bar g|_{L_2(0,L/2)}^2 - \omega_+ \bar g(0) = \varepsilon \omega_+^2  +  \frac{(1 - \omega_+ \varepsilon)^2}{ L/2 - \varepsilon}  - \omega_+.
\end{align}
Since $\omega_+ >0$ is fixed, we may choose $\varepsilon > 0$ so small, such that $\varepsilon \omega_+^2 \leq \omega_+/2$. Then choosing $L = L(\omega_+ ) > 0$ sufficiently large we obtain that $| \partial_x \bar g|_{L_2(0,L/2)}^2 - \omega_+ \bar g(0) < 0$. We can then follow the lines of the proof of (2) to conclude (4).
\end{proof}
\begin{remark}
For monotonicity considerations of the spectral properties we refer to \cite{garckeitokosaka}.
\end{remark}
\section{Curved stationary solutions}
In this section we consider stationary solutions $\Sigma_*$ with constant curvature $\kappa_* = - 1/R$, for some $R >0$. In particular, $\Sigma_*$ is part of a circle. We can therefore introduce a parametrization by arc length,
\begin{equation}
\psi : (0,l) \pfeil \Sigma_*, \; \sigma \pfeil \psi (\sigma),
\end{equation}
where $l > 0$ is the length of the curve and $\sigma$ the arc length parameter. Note that $l < 2\pi R = 2\pi/|\kappa_*|$. Note that this induces an extra restriction on $\kappa_*$ and $l$,
\begin{equation} \label{3493095645064596} \tag{5.1}
| \kappa_* | l < 2 \pi.
\end{equation}
Corresponding to (3.1) we now want to make a linear stability analysis for 
 \begin{equation} \label{345dfg45dfudf5} \tag{5.2}
\begin{alignedat}{2}
\partial_t \rho &= - \ljump n_{\Sigma_*} \cdot \nabla \mu \rjump, \quad\quad & \text{on } \Sigma_*, \\
\mu|_{\Sigma_*} &= \partial_\sigma \partial_\sigma \rho + \kappa_*^2 \rho , \quad\quad\quad &\text{on } \Sigma_*, \\
\Delta \mu &= 0, &\text{in } \Omega \backslash \Sigma_*, \\
n_{\partial\Omega} \cdot \nabla \mu|_{\partial \Omega} &= 0, &\text{on } \partial\Omega \backslash \partial\Sigma_*, \\
\partial_\sigma \rho (0) &=  - \omega_1 \rho (0), \\
\partial_\sigma \rho (l) &=  \omega_2 \rho (l), \\
\rho (0) &= \rho_0, &\text{on }\Sigma_* .
\end{alignedat}
\end{equation}
Note that by some abuse of notation we may identify $\sigma \in (0,l)$ and $\psi(\sigma) \in \Sigma_*$, since there is no danger of confusion.

Let us rewrite (5.2) again as an abstract evolution equation, cf. \cite{mulsekpaper12}. Let $3/2 < q < 2$, $
X_0 := W^{1-1/q}_q(0,l)$, and $ X_1 := W^{4-1/q}_q(0,l)$.
Define now a linear operator $A : D(A) \subset X_1 \pfeil X_0$ by means of
$A \rho :=  B T_0 ( \partial_\sigma \partial_\sigma \rho + \kappa_*^2 \rho )$, where
 $B u := \ljump n_{\Sigma_*} \cdot \nabla u \rjump$ and $T_0 v$ is the unique solution of the two-phase elliptic problem \eqref{345fdfgdfgh4g5}.

The domain of $A$ is thereby given by
\begin{equation} \label{345j3kh4jk5645} \tag{5.3}
D(A) := X_1 \cap \{ \rho : \partial_\sigma \rho (0) = - \omega_1 \rho(0), \; \partial_\sigma \rho (l) = \omega_2 \rho(l) \}.
\end{equation}
We can then rewrite (5.2) as the abstract evolutionary problem
\begin{equation} \label{kjlhfghfgh} \tag{5.4}
\dot \rho + A \rho = 0, \; t > 0, \quad \rho(0) = \rho_0.
\end{equation}
Let again $p \in (6,\infty)$, $q \in (19/10,2) \cap (2p/(p+1), 2)$ and $X_\gamma := B^{4-1/q-3/p}_{qp} (0,l)$.

We start with a positive result on exponential stability for \eqref{kjlhfghfgh} of the trivial solution $\rho_* = 0$.
\begin{theorem} \label{jkh3dfgdfg45} 

Let $l > 0$ be fixed. 
Then there is some $\delta = \delta(l) >0$, such than whenever $|\kappa_*| \in (0, \delta)$ and $\omega_1, \omega_2 \in (-\infty, \delta)$, the trivial equilibrium
 $\rho_* = 0$ is normally stable, that is:
\begin{enumerate}
\item $A$ has maximal $L_p$-regularity.
\item The set of equilibria of (5.4) is the kernel of $A$, which has finite dimension $m \in \mathbb N \cup \{ 0 \}, m < \infty$.
\item The eigenvalue zero is semi-simple, $X_0 = N(A) \oplus R(A)$.
\item The spectrum satisfies $\sigma (-A) \backslash \{ 0 \} \subset \mathbb C_- := \{ z \in \mathbb C : \operatorname{Re} z < 0 \} $.
\end{enumerate}
In particular, $\rho_* = 0$ is stable in $X_\gamma$ and there is some $\delta_1 > 0$, such that if $|\rho_0 |_{X_\gamma} \leq \delta_1$ the unique solution to (5.4) with respect to the initial value $\rho_0$ exists globally in time,
$$ \rho \in W^1_p(\R_+ ; X_0) \cap L_p(\R_+ ; D(A)),$$ and converges to some equilibrium solution in $X_\gamma$ at an exponential rate.
\end{theorem}
\begin{proof}
For any $\kappa_*$ constant we note that the term $\kappa_* \rho$ is a compact perturbation of $\partial_\sigma \partial_\sigma \rho$ in $W^{2-1/q}_q(0,l)$, whence $A$ has maximal $L_p$-regularity by a perturbation argument, cf. \cite{mulsekpaper12}.
Let us now characterize the kernel of $A$. Since the domain $D(A)$ compactly embeds into $X_0$, the resolvent of $A$ is compact. The spectrum then consists solely of isolated eigenvalues of finite multiplicity. In particular, the kernel, if it is nontrivial, has finite dimension $m < \infty$, cf. \cite{engelnagel}, \cite{lunardioptimal}, \cite{lunardiinterpol}. Pick some $\rho \in D(A)$ such that $A\rho = 0$. Then the solution of the corresponding elliptic problem is constant, hence $\partial_ \sigma \partial_\sigma \rho + \kappa_*^2 \rho$ is constant. 
Therefore the kernel of $A$ is given by the solutions $\rho$ of
\begin{equation} \begin{alignedat}{2}
\partial_\sigma \partial_\sigma \rho + \kappa_*^2 \rho &= c, \quad &\text{on } (0,l), \\
\partial_\sigma \rho (0) &= - \omega_1 \rho(0), \quad\quad \\
\partial_\sigma \rho (l) &= \omega_2 \rho(l),
\end{alignedat}
\end{equation}
where $c$ is any constant $c \in \R$. The next thing we show is that zero is semi-simple. By Remark A.2.4 in \cite{lunardioptimal} it suffices to prove that $N(A^2) = N(A)$. To this end pick some $\rho \in N(A^2)$. Let $\rho_1 := A\rho$. Then $A\rho_1 = 0$ and hence $\rho_1 \in N(A) \cap R(A)$. Note that then necessarily $\rho_1$ is mean value free, $\int_0^l \rho_1 = 0$. 
Since $\rho_1$ also belongs to the kernel of $A$,
\begin{equation}  \tag{5.5} \begin{alignedat}{2} \label{304956fff85645}
\partial_\sigma \partial_\sigma \rho_1 + \kappa_*^2 \rho_1 &= c_1, \quad &\text{on } (0,l), \\
\partial_\sigma \rho_1 (0) &= - \omega_1 \rho_1(0), \quad\quad \\
\partial_\sigma \rho_1 (l) &= \omega_2 \rho_1(l),
\end{alignedat}
\end{equation}
for some constant $c_1$. Note that $c_1$ is determined by $\rho_1$. Since $\rho_1$ is mean value free, we can test $(5.5)_1$ with $\rho_1$ to the result
\begin{equation}
\int_0^l \partial_\sigma \partial_\sigma \rho_1 \rho_1 + \kappa_*^2 \int_0^l \rho_1 \rho_1 = c_1 \int_0^l \rho_1 = 0.
\end{equation}
An integration by parts then gives
\begin{equation}
- \int_0^l  |\partial_\sigma \rho_1|^2 + \kappa_*^2 \int_0^l | \rho_1|^2 + \omega_1 \rho_1(0)^2 + \omega_2 \rho_1(l)^2 = 0.
\end{equation}
Since $\rho_1$ is mean value free, we can use Poincaré-Wirtinger inequality to find some constant $c_0 = c_0(l) > 0$, such that
\begin{equation}
- c_0  | \rho_1|_{H^1}^2 + \kappa_*^2  | \rho_1|_{L^2}^2 + \omega_1 \rho_1(0)^2 + \omega_2 \rho_1(l)^2 \geq 0.
\end{equation}
In particular, if $\kappa_*^2$ is sufficiently small and $\omega_1, \omega_2$ are negative or positive but small, the second, third and fourth term may be absorbed by the first one and we obtain
\begin{equation}
- \tilde c_0  | \rho_1|_{H^1}^2  \geq 0,
\end{equation}
for some $\tilde c_0 > 0$. Hence $\rho_1  = 0$, which implies $A\rho = \rho_1 = 0$ and $\rho \in N(A)$. This shows zero is a semi-simple eigenvalue.

Let us now consider the general eigenvalue problem $\lambda \rho = -A\rho$ for some $\rho \in D(A)$, which reads as
 \begin{equation} \label{3fghdfgdfghfghh5} \tag{5.6}
\begin{alignedat}{2}
\lambda \rho &= - \ljump n_{\Sigma_*} \cdot \nabla \mu \rjump, \quad\quad & \text{on } \Sigma_*, \\
\mu|_{\Sigma_*} &= \partial_\sigma \partial_\sigma \rho + \kappa_*^2 \rho , \quad\quad\quad &\text{on } \Sigma_*, \\
\Delta \mu &= 0, &\text{in } \Omega \backslash \Sigma_*, \\
n_{\partial\Omega} \cdot \nabla \mu|_{\partial \Omega} &= 0, &\text{on } \partial\Omega \backslash \partial\Sigma_*, \\
\partial_\sigma \rho (0) &=  - \omega_1 \rho (0), \\
\partial_\sigma \rho (l) &=  \omega_2 \rho (l).
\end{alignedat}
\end{equation}
Testing $(5.6)_1$ with $\partial_\sigma \partial_\sigma \rho + \kappa_*^2 \rho$ in $L_2$ and invoking boundary and transmission conditions gives
\begin{equation} \label{3456jkh45j6hj5k4hrth} \tag{5.7}
\lambda \left[ | \partial_\sigma \rho |_{L^2}^2 - \kappa_*^2 | \rho |_{L^2}^2 - \omega_1 \rho (0)^2 - \omega_2 \rho (l)^2 \right] + | \nabla \mu |_{L^2}^2 = 0.
\end{equation}
If $\lambda \not= 0$, any eigenfunction $\rho$ to an eigenvalue $\lambda$ is necessarily mean value free, whence again Poincaré-Wirtinger inequality gives that $$ \left[ | \partial_\sigma \rho |_{L^2}^2 - \kappa_*^2 | \rho |_{L^2}^2 - \omega_1 \rho (0)^2 - \omega_2 \rho (l)^2 \right]  \geq 0,$$ provided $| \kappa_*| \in [0,\delta)$ and $\omega_1, \omega_2 \in (-\infty, \delta) $ for $\delta > 0$ sufficiently small. Equation (5.7) then gives that $\lambda$ is real and $\lambda \leq 0$. Hence (3) follows. The generalized principle of linearized stability of Pr\"uss, Simonett, and Zacher \cite{pruessmulsek} then gives the result.
 \end{proof}
 Let us show instability results for the evolution equation (5.4).
 
 \begin{theorem}
 Let $A,X_0,X_\gamma,X_1$ be as above in (5.3).
 \begin{enumerate}
 \item For fixed $l > 0$ and any small $\kappa_*$, there is some $K = K(l,\kappa_*) > 0$ such that if $\omega_1 = \omega_2 \geq K$, the trivial solution $\rho_* = 0$ is unstable in $X_\gamma$.
 \item For fixed $l > 0$ and any $\omega_1 = \omega_2$ small, there is some $K = K(l,\omega_1) > 0$ such that if $\kappa_*^2 \geq K$, the trivial solution $\rho_* = 0$ is unstable in $X_\gamma$.
 \item For any $\kappa_*$ and $\omega_1 = \omega_2$ small, there is some $K = K(\kappa_*, \omega_1) > 0$ such that if $l \geq K$, the trivial solution $\rho_* = 0$ is unstable in $X_\gamma$.
 \item In (2) and (3) the constant $K$ is not too large to violate the geometric condition between length and curvature of a circle (5.1), that means there are $(\kappa_*, l)$ fulfilling (2) or (3) which at the same time fulfil $|\kappa_*| l < 2 \pi$.
 \end{enumerate}
 In particular, in any of these cases, $\sigma(-A) \cap [ \zeta + i \mathbb R ] = \emptyset$ and $\sigma(-A) \cap \{ z \in \mathbb C : \operatorname{Re} z > \zeta \} \not = \emptyset$ for some $\zeta \in \R, \zeta \geq 0$.
 \end{theorem}
 \begin{proof}
 By the compact embedding $D(A) \into \into X_0$ we know that $A$ has a compact resolvent. Hence the spectrum of $A$ is isolated, consists only of eigenvalues and each eigenvalue has finite multiplicity. Furthermore, any eigenvalue $\lambda$ is real and satisfies
 \begin{equation} \label{3456jkh4dfgdfgdfgdfgdfg5j6hj5k4hrth}
\lambda \left[ | \partial_\sigma \rho |_{L^2}^2 - \kappa_*^2 | \rho |_{L^2}^2 - \omega_1 \rho (0)^2 - \omega_2 \rho (l)^2 \right] + | \nabla \mu |_{L^2}^2 = 0,
\end{equation}
cf. (5.7), where $\mu = T_0(\partial_\sigma \partial_\sigma \rho + \kappa_*^2 \rho)$ and $\rho$ is a corresponding eigenfunction to $\lambda$.

We now follow the lines of the proof of (2) in Theorem 4.2. For $\lambda > 0$ we can rewrite the eigenvalue problem $\lambda \rho = A \rho$ as
\begin{equation} \label{3405634596nj456} \tag{5.8}
\lambda \rho - D_{MS} (I - P_0) S \rho = 0,
\end{equation}
where $D_{MS}$ is as before the corresponding Dirichlet-to-Neumann operator with inverse $N_{MS} = [D_{MS}]^{-1}$, $P_0 f$ the mean value of $f$, and $S\rho := \partial_\sigma \partial_\sigma \rho + \kappa_*^2 \rho$ with domain $D(S) := D(A)$. We can then extend the operators in a natural way and rewrite (5.8) as
\begin{equation} \label{34056345f96nj456}
\lambda N_{MS} \rho -  (I - P_0) S \rho = 0.
\end{equation}
Define $B_\lambda := \lambda N_{MS} - (I-P_0) S$ with natural domain $$D(B_\lambda) := H^2_{2} (0,l) \cap \{ \rho : \int_0^l \rho = 0, \; \partial_\sigma \rho (0) = - \omega_1 \rho (0), \; \partial_\sigma \rho (l) = - \omega_1 \rho (l) \}.$$
Let us show that there is some $\lambda_0 > 0$ such that $B_\lambda$ is positive definite on $L_{2,(0)}$ for all $\lambda \geq \lambda_0$. Since $N_{MS}$ is positive definite on $L_{2,(0)}$, cf. \cite{rtipaper}, and $\rho$ is mean value free,
\begin{align}
(B_\lambda \rho | \rho)_2 &= \lambda (N_{MS} \rho | \rho)_2 - ( (I-P_0) (\partial_\sigma \partial_\sigma \rho + \kappa_*^2 \rho ) | \rho)_2 \\
&= \lambda (N_{MS} \rho | \rho)_2 - ( \partial_\sigma \partial_\sigma \rho + \kappa_*^2 \rho  | \rho)_2 \\
&\geq \lambda d_0 | \rho |_2^2 + | \partial_\sigma \rho |_2^2 - \omega_1 \rho(0)^2 - \omega_2 \rho (l)^2 - \kappa_*^2 | \rho |_2^2,
\end{align}
for some $d_0 > 0$, for any $\rho \in D(B_\lambda)$. Lemma 4.3 then gives that for any $\omega_1, \omega_2,$ and $ \kappa_*$, the last three terms may be absorbed if $\lambda \geq \lambda_0$ for some $\lambda_0 =  \lambda_0 (\omega_1,\omega_2, \kappa_*) > 0$. Hence $B_\lambda$ is positive definite on $L_{2,(0)}$ for all $\lambda \geq \lambda_0$.

It remains to construct a function $\bar \rho \in D(B_\lambda)$, such that $(B_\lambda \bar \rho | \bar \rho )_2 < 0$, if $\lambda > 0$ is sufficiently small. Since $\lambda (N_{MS} \rho | \rho)_2 \pfeil 0$ as $\lambda \pfeil 0$, it is enough to find some $\bar \rho \in D(B_\lambda)$ such that
\begin{equation}  \label{34598067459086450896gjfkhgjfk} \tag{5.9}
 | \partial_\sigma \bar \rho |_2^2 - \omega_1 \bar \rho(0)^2 - \omega_2 \bar \rho (l)^2 - \kappa_*^2 | \bar \rho |_2^2 <0.
 \end{equation}
 We now want to find such $\bar \rho$ in all three cases stated in the theorem. To this end we start again with the prototype introduced in (4.15). Let $\varepsilon_1 > 0$ small. Define
 $\bar g : [0,l] \pfeil \R, \sigma \mapsto \bar g(\sigma)$ by means of
\begin{equation} 
\bar g(\sigma) := \begin{cases}
1- \omega_1 \sigma, & \sigma \in [0, \varepsilon_1], \\
1- \omega_1 \varepsilon_1 - (1-\omega_1\varepsilon_1)(\sigma-\varepsilon_1)/(l/2-\varepsilon_1), & \sigma \in [\varepsilon_1, l/2], \\
-g(\sigma-l), & \sigma \in [l/2, l].
\end{cases}
\end{equation}
Note that $\bar g(0) = 1$, $\bar g(l) = -1$, $\bar g$ fulfils the boundary conditions, is mean value free, and piecewise smooth and continuous, hence in $H^1_2 (0,l)$.
We then explicitly calculate
\begin{align*}
& | \partial_\sigma \bar g |_{L_2(0,l/2)}^2 - \omega_1 \bar g(0)^2 - \kappa_*^2 | \bar g |_{L_2(0,l/2)}^2  = \\
&= \varepsilon_1 \omega_1^2 + \frac{ (1- \omega_1 \varepsilon_1)^2}{l/2 - \varepsilon_1} - \omega_1 - \kappa_*^2 \left[ \frac{\varepsilon_1}{3} ( 3 - 3 \omega_1 \varepsilon_1 + \omega_1^2 \varepsilon_1 ) + \frac{1}{6} (1- \omega_1 \varepsilon_1)^2 (l - 2 \varepsilon_1) \right].
\end{align*}
Note that by symmetry it suffices to calculate the expressions on $(0,l/2)$.
We now let formally $\varepsilon_1 \pfeil 0$. The expression on the right hand side then converges to
\begin{equation} \label{3489053485} \tag{5.10}
\frac{2}{l} - \omega_1 - \kappa_*^2 \frac{l}{6}.
\end{equation}
We now distinguish the three cases in the theorem.
\begin{enumerate}
\item Here we fix $l > 0$ and $\kappa_*$. It is clear that there is some $K>0$ such that the expression in (5.10) gets strictly negative if $\omega_1 \geq K$. It even holds that $2/l - \omega_1 - \kappa_*^2 l/6 \pfeil -\infty$ if $\omega_1 \pfeil \infty$.
\item Here we fix $l > 0$ and $|\omega_1|$ small. Note that in this case there is a geometric condition, $|\kappa_*| l < 2 \pi$, so we can not choose $|\kappa_*|$ arbitrarily large. However, taking the limit as $|\kappa_*| \pfeil 2\pi/l$ of (5.10), we obtain
\begin{equation} \label{348hjkfghjk85}
 \lim_{ |\kappa_*| \pfeil 2 \pi/l} \left[ \frac{2}{l} - \omega_1 - \kappa_*^2 \frac{l}{6} \right] = \frac{2}{l} \left[ 1 - \frac{\pi^2}{3} \right] - \omega_1 < 0,
\end{equation}
provided $|\omega_1|$ is small enough.
\item In this case we fix $\kappa_*$ and $|\omega_1|$ small. Again we have to fulfil the relation $|\kappa_*| l < 2 \pi$. Taking limits $l \pfeil 2\pi/|\kappa_*|$,
\begin{equation} \label{348hjkfghjk85}
 \lim_{l \pfeil 2 \pi/|\kappa_*|} \left[ \frac{2}{l} - \omega_1 - \kappa_*^2 \frac{l}{6} \right] = |\kappa_*| \left[ \frac{1}{\pi} - \frac{\pi}{3} \right] - \omega_1 < 0,
\end{equation}
provided again $|\omega_1|$ is small.
\end{enumerate}
This way we now obtain the following result in all three cases: By choosing $\varepsilon_1 > 0$ very small, we can construct $\bar g$ as above such that the strict inequality (5.9) holds true. Since $\bar g$ is only $H^1_2$ and not $H^2_2$ we need to approximate $\bar g$ by a more regular function $\bar \rho$, which then belongs to the domain $D(B_\lambda)$ and also fulfils the strict inequality (5.9). This then shows that $(B_\lambda \bar \rho | \bar \rho )_2 < 0$ for some $\bar \rho \in D(B_\lambda)$ if $\lambda > 0$ is sufficiently small, whence $B_\lambda$ is not positive definite for this small $\lambda > 0$. Following the arguments of \cite{wilkehabil} we then obtain the existence of a positive eigenvalue. 

The fact that then $\rho_* = 0$ is unstable in $X_\gamma$ follows from the generalized principle of linearized stability, cf. \cite{pruessbuch}.
 \end{proof}
 
\section{summary on linearized stability and instability}
In this section we shall summarize the results on linearized stability.
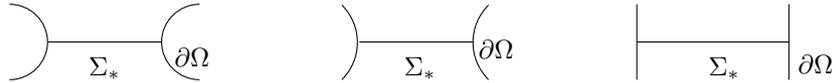
\begin{figure}[ht]
\begin{tikzpicture}
\draw (-1.6,0) arc (270:90:0.5cm);
\draw (-4.1,0) arc (-90:90:0.5cm);
\draw (-3.6,0.5) -- (-2.1,0.5);

\draw (2.2,0) arc (225:135:0.7cm);
\draw (0.27,0) arc (-45:45:0.7cm);
\draw (0.5,0.5) -- (2,0.5);

\draw (4.15,0) -- (4.15,1);
\draw (6.15,0) -- (6.15,1);
\draw (4.15,0.5) -- (6.15,0.5);

\draw (5.3,0.15) node {$\Sigma_\ast$};
\draw (1.3,0.15) node {$\Sigma_\ast$};
\draw (-2.85,0.15) node {$\Sigma_\ast$};
\draw (2.3,0.4) node {$\partial\Omega$};
\draw (-1.7,0.3) node {$\partial\Omega$};
\draw (6.5,0.2) node {$\partial\Omega$};
\end{tikzpicture}
\caption{$\kappa_* = 0$. Exponential stability for all $\omega_1, \omega_2 \leq 0$ regardless of $L > 0$.}
\label{figucxycvgg}
\end{figure}
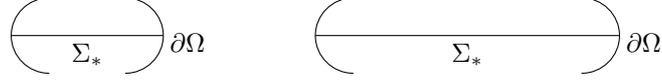
\begin{figure}[ht]
\begin{tikzpicture}

\draw (8.5,0) arc (270:90:0.5cm);
\draw (9.5,0) arc (-90:90:0.5cm);
\draw (8,0.5) -- (10,0.5) ;

\draw (12.5,0) arc (270:90:0.5cm);
\draw (15.5,0) arc (-90:90:0.5cm);
\draw (12,0.5) -- (16,0.5) ;

\draw (9,0.25) node {$\Sigma_\ast$};
\draw (10.32,0.4) node {$\partial\Omega$};
\draw (14,0.25) node {$\Sigma_\ast$};
\draw (16.32,0.4) node {$\partial\Omega$};
\end{tikzpicture}
\caption{$\kappa_* = 0$ and fixed $\omega_1 = \omega_2 > 0$. Exponential stability for small $L>0$, instability for large $L>0$.}
\label{figucxcvxcg}
\end{figure}
\begin{figure}[ht]
\begin{tikzpicture}

\draw (0.1,0) arc (200:160:1.4cm);
\draw (1.9,0) arc (-20:20:1.4cm);
\draw (0,0.5) -- (2,0.5) ;

\draw (4.5,0) arc (270:90:0.5cm);
\draw (5.5,0) arc (-90:90:0.5cm);
\draw (4,0.5) -- (6,0.5) ;

\draw (1,0.25) node {$\Sigma_\ast$};
\draw (2.32,0.4) node {$\partial\Omega$};
\draw (5,0.25) node {$\Sigma_\ast$};
\draw (6.32,0.4) node {$\partial\Omega$};
\end{tikzpicture}
\caption{$\kappa_* = 0$ and fixed $L>0$. Exponential stability for small $\omega_1 = \omega_2 > 0$, instability for large $\omega_1 = \omega_2$.}
\label{figcvvvccvcg}
\end{figure}
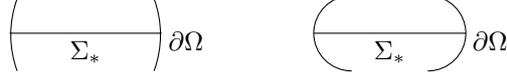

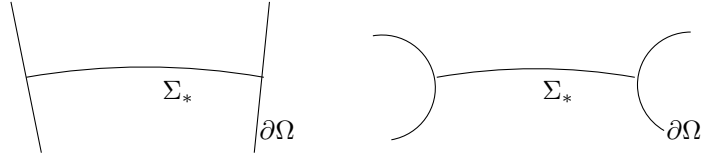
\begin{figure}[ht]
\begin{tikzpicture}

\draw (0,0) arc (100:80:9cm);
\draw (2,-0.2) node {$\Sigma_\ast$};
\draw (3.3,-0.7) node {$\partial\Omega$};
\draw (0.2,-1) -- (-0.2,1);
\draw (3,-1) -- (3.2,1);

\draw (5.4,0) arc (100:80:7.5cm);
\draw (7,-0.2) node {$\Sigma_\ast$};
\draw (8.65,-0.7) node {$\partial\Omega$};
\draw (8.74,0.6) arc (90:240:0.7cm);
\draw (4.56,0.55) arc (100:-80:0.7cm);

\end{tikzpicture}
\caption{Fixed $l >0$ and small $\kappa_* \not=0 $. Exponential stability for $\omega_1, \omega_2 \leq 0$.}
\label{figuvcvcbcvcvg}
\end{figure}
\begin{figure}[ht]
\begin{tikzpicture}

\draw (0,0) arc (100:80:9cm);
\draw (2,-0.2) node {$\Sigma_\ast$};
\draw (3.3,-0.7) node {$\partial\Omega$};
\draw (0,0.6) arc (175:215:2cm);
\draw (3.1,0.6) arc (12:-28:2cm);

\draw (5,0) arc (100:80:9cm);
\draw (7,-0.2) node {$\Sigma_\ast$};
\draw (8.3,-0.5) node {$\partial\Omega$};
\draw (5.4,0.7) arc (110:270:0.7cm);
\draw (7.75,0.7) arc (65:-85:0.7cm);

\end{tikzpicture}
\caption{Fixed $l>0$ and small $\kappa_* \not=0$. Exponential stability for $\omega_1, \omega_2 >0$ small. Instability for $\omega_1,\omega_2 >0$ large.}
\label{figcvbfgbfgbg}
\end{figure}
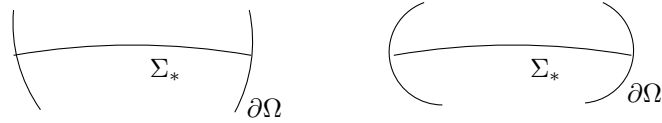
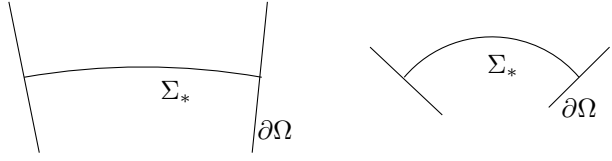
\begin{figure}[ht]
\begin{tikzpicture}

\draw (0,0) arc (100:80:9cm);
\draw (2,-0.2) node {$\Sigma_\ast$};
\draw (3.3,-0.7) node {$\partial\Omega$};
\draw (0.2,-1) -- (-0.2,1);
\draw (3,-1) -- (3.2,1);

\draw (5,0) arc (140:40:1.5cm);
\draw (6.3,0.15) node {$\Sigma_\ast$};
\draw (7.3,-0.4) node {$\partial\Omega$};
\draw (5.5,-0.5) -- (4.55,0.4);
\draw (6.9,-0.4) -- (7.7,0.4);

\end{tikzpicture}
\caption{Fixed $l >0$ and $\omega_1, \omega_2 = 0$. Exponential stability for $\kappa_*^2$ small. Instability for $\kappa_*^2$ large.}
\label{figtbcvbfbgg}
\end{figure}

\begin{figure}[!]
\begin{tikzpicture}

\draw (0,0) arc (160:20:1cm);
\draw (1.3,0.85) node {$\Sigma_\ast$};
\draw (2,-0.25) node {$\partial\Omega$};
\draw (-0.5,0.15) -- (0.5,-0.15);
\draw (1.5,-0.15) -- (2.5,0.15);

\draw (5,0) arc (250:-70:1cm);
\draw (6,1.4) node {$\Sigma_\ast$};
\draw (5.2,-0.4) node {$\partial\Omega$};
\draw (4.8,-0.5) -- (5.2,0.4);
\draw (5.9,-0.4) -- (5.5,0.4);

\end{tikzpicture}
\caption{Fixed $\kappa_* \not= 0$ and $\omega_1, \omega_2 = 0$. Exponential stability for $l > 0$ small. Instability for $l>0$ large.}
\label{figcvbcvbcvbg}
\end{figure}
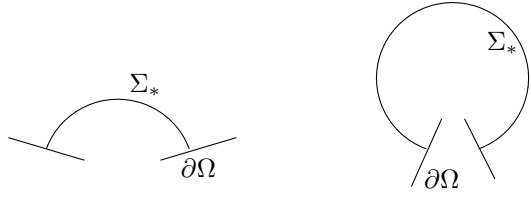
\section{Nonlinear stability}
In this section we 
show a first nonlinear stability result for the case of a flat stationary solution. We
are concerned with the full transformed nonlinear problem (2.1) for the height function, which reads as
\begin{equation}  \tag{7.1} \label{378645087645087405}
\begin{alignedat}{2}
 \partial_t h &= - a(h)\ljump n_{\Gamma_h} \cdot \nabla_h \eta \rjump , \quad\quad &&\text{on } \Sigma, \\
\eta|_{\Sigma} &= K(h), && \text{on } \Sigma, \\
\Delta_h \eta &= 0, && \text{in } \Omega \backslash \Sigma, \\
n_{\partial\Omega}^h \cdot \nabla_h \eta|_{\partial\Omega} &= 0, && \text{on } \partial\Omega, \\
n_{\partial\Omega}^h \cdot n_{\Gamma_h} &= 0, && \text{on } \partial\Sigma, \\
h|_{t=0} &= h_0, && \text{on } \Sigma,
\end{alignedat}
\end{equation}
where
$ a (h) (t) = a (h(t))$, $ a ( h(t))$ depends smoothly on $h(t)$, and $a(0)=1$. For a complete deduction of (7.1) we refer to \cite{mulsekpaper12}.

We want to rewrite the full nonlinear system (7.1) as an abstract evolution equation in an $L_p$-setting. By the generalized principle of linearized stability \cite{pruessmulsek} we then deduce nonlinear stability and convergence to an equilibrium solution at exponential rate.

First we need to analyze the boundary condition $(7.1)_5$. Since we work in two space dimensions, the condition $(7.1)_5$ can be rewritten as
\begin{equation}
n_{\partial \Omega} ( \tilde \Theta (h(t)) \cdot R \tau_h = 0,
\end{equation}
where $\tilde \Theta ( h(t))$ is defined by means of $\tilde \Theta (h(t)) (x) := \Theta_h^t (x)$, $R$ is the rotation of $90$ degrees counterclockwise, and $\tau_h$ is the tangent vector to the graph of $h$. In particular,
\begin{equation}
\tau_h (p) = \partial_p X (p, h(p)) + \partial_w X (p, h(p)) \partial_p h(p), \quad p \in \Sigma.
\end{equation}
Clearly, $R\tau_h$ is parallel to the unit normal $n_{\Gamma_h}$. Then $(7.1)_5$ is equivalent to
\begin{equation}
n_{\partial \Omega} ( \tilde\Theta (h(t)) \cdot R \partial_p X (h(t)) + n_{\partial \Omega} ( \tilde\Theta (h(t)) \cdot R \partial_w X (h(t)) \partial_p h (t)= 0,
\end{equation}
where we surpress the dependence of $p \in \Sigma$ in the notation. To economize notation, define
\begin{equation}
G(h(t)) := n_{\partial \Omega} ( \tilde\Theta (h(t)) \cdot R \partial_p X (h(t)), \quad H (h(t)) := n_{\partial \Omega} ( \tilde\Theta (h(t)) \cdot R \partial_w X (h(t)).
\end{equation}
Note that by properties of the curvilinear coordinate system $X$,
\begin{equation}
G,H \in C^\infty(X_j ; \R), \; j \in \{ 0,\gamma,1 \}, \quad G (0) = 0, \; G' (0) = 0, \;|H(0)| = 1, \; \text{on } \partial \Sigma.
\end{equation}
Since we are interested in stability properties around $h = 0$, we can assume that $|h|_{X_\gamma} \leq \delta_1$ for some uniform $\delta_1 > 0$ such that $|H(h)| \geq 1/2$ for all $|h|_{X_\gamma} \leq \delta_1$. Therefore we can define
\begin{equation}
\mathcal G (h) := - \frac{ G (h)}{ H(h)}, \quad h \in X_\gamma, |h|_{X_\gamma} \leq \delta_1.
\end{equation}
Then $\mathcal G \in C^\infty(X_j; \R), j \in \{0,\gamma \}$, $\mathcal G$ is also quadratic in zero, and the semilinear boundary condition 
\begin{equation}
\partial_p h(t) = \mathcal G (h(t)), \quad \text{on } \partial \Sigma,
\end{equation}
is equivalent to $(7.1)_5$, for $|h|_{X_\gamma} \leq \delta_1$.

Let us introduce notation. 
Define the operator $\mathcal C$ by means of $$\mathcal C (h)[g] := - a(h)  \ljump n_{\Gamma_h} \cdot  \nabla_h g \rjump  .$$
We remark that in particular $\mathcal C(0)[g] = - \ljump n_\Sigma \cdot \nabla g \rjump$.
Denote by $T^f g$ the unique solution of the two-phase elliptic problem
\begin{align}
\eta|_{\Sigma} &= g, && \text{on } \Sigma, \\
\Delta_f \eta &= 0, && \text{in } \Omega \backslash \Sigma, \\
n_{\partial\Omega}^f \cdot \nabla_f \eta|_{\partial\Omega} &= 0, && \text{on } \partial\Omega, 
\end{align}
cf. Appendix A in \cite{mulsekpaper12}.
Furthermore, define the boundary operator $\mathcal B$ by means of $\mathcal B g := \partial_p g$. Then (7.1) can be rewritten as
\begin{equation} \tag{7.2} \label{rgjhjrkehgf87405}
\begin{alignedat}{2}
\partial_t h (t)&= - \mathcal C (h(t)) [ T^{h(t)} K(h(t))] , \quad\quad\quad &&\text{on } \Sigma, t > 0, \\
\mathcal B h(t) &= \mathcal G (h(t)), &&\text{on } \partial\Sigma, t > 0, \\
h (0) &= h_0, && \text{on } \Sigma, t = 0.
\end{alignedat}
\end{equation}
Let $\mathcal O$ be a sufficiently small neighbourhood of zero in $X_\gamma$. Then
we can decompose 
\begin{equation}
K(h) = P(h)h + Q(h), \quad h \in \mathcal O \cap X_1,
\end{equation}
where $P \in C^1 (\mathcal O ; \mathcal B(X_1 ; W^{2-1/q}_q(\Sigma)))$, $Q \in C^1( \mathcal O ; W^{2-1/q}_q(\Sigma))$, cf. \cite{mulsekpaper12}, \cite{eschersimonett}. Furthermore, since $\Sigma$ is flat, $P(0) = \Delta_\Sigma$, the derivative of $K$ at point zero is given by $[h \mapsto \Delta_\Sigma h]$, cf. Lemma 6.1 in \cite{mulsekpaper12}, and $Q$ is quadratic in zero.

Define now the operator $A$ by means of $A(h) := \mathcal C(h)T^h P(h)$. Then (7.2) is equivalent to the evolutionary problem
\begin{equation} \tag{7.3} \label{39049538475384534}
\begin{alignedat}{2}
\partial_t h (t) + A (h(t)) h(t) &= F(h(t)) , \quad\quad\quad &&\text{on } \Sigma, t > 0, \\
\mathcal B h(t) &= \mathcal G (h(t)), &&\text{on } \partial\Sigma, t > 0, \\
h (0) &= h_0, && \text{on } \Sigma, t = 0,
\end{alignedat}
\end{equation}
where $F(h) := -\mathcal C(h)T^h Q(h)$. In particular, let us note that
\begin{equation} \label{34095734875390485} \tag{7.4}
A \in C^1(\mathcal O; \mathcal B(X_1 ;X_0)), \quad F \in C^1(\mathcal O ; X_0 ),
\end{equation}
cf. \cite{mulsekpaper12}, \cite{eschersimonett}.
We now want to apply the generalized principle of linearized stability \cite{pruessmulsek} to problem (7.3). To this end we need to rewrite (7.3) as a single local-in-time evolution equation for $h$, cf. \cite{apace}. 

Recall that the trace space is given by $X_\gamma = B^{4-1/q-3/p}_{qp} (\Sigma)$.
Let $E$ be the corresponding extension operator to $\mathcal B$ in space, satisfying
\begin{equation}
Eg \in X_\gamma, \quad \mathcal B Eg = g, \; \text{on } \partial \Sigma, \quad Eg = 0, \; \text{on } \partial \Sigma.
\end{equation}
Recall that in two space dimensions, $\partial \Sigma$ only consists of two isolated points. Having this operator at hand, we define $\tilde h(t) := E \mathcal G (h(t))$, which is a function in $C^0([0,T];X_\gamma)$. Set $v :=  h- \tilde h$.
Then we can rewrite (7.3) as
\begin{equation} \label{3904fghjfgh9538475384534}
\begin{alignedat}{2}
\partial_t v (t) + A (v(t) + \tilde h(t)) v(t) &= F(v(t)+\tilde h(t)) - \partial_t \tilde h(t) \\ &\quad- A( v(t) + \tilde h (t)) \tilde h(t) , \quad\quad\quad &&\text{on } \Sigma, t > 0, \\
\mathcal B v(t) &=0, &&\text{on } \partial\Sigma, t > 0, \\
v (0) &= v_0 := h_0 - \tilde h(0), && \text{on } \Sigma, t = 0.
\end{alignedat}
\end{equation}
Now,
\begin{equation}
\tilde h (t) =  E \mathcal G (h(t)) = E \mathcal G (v(t) + \tilde h(t) ) =E \mathcal G (v(t) ),
\end{equation}
since $\tilde h$ vanishes on the boundary $\partial \Sigma$. Therefore (7.3) can be recast as
\begin{equation} \label{390495384753845dfgdfgdfg34} \tag{7.5}
\begin{alignedat}{2}
\partial_t v (t) + \tilde A (v(t))[ v(t) ]&= \tilde F(v(t)) , \quad\quad\quad &&\text{on } \Sigma, t > 0, \\
\mathcal B [v(t)] &=0, &&\text{on } \partial\Sigma, t > 0, \\
v (0) &= v_0, && \text{on } \Sigma, t = 0,
\end{alignedat}
\end{equation}
where
\begin{align}
\tilde A (v) v &:= A (v + E\mathcal G (v) ) v, \\
\tilde F(v) &:=  F(v + E\mathcal G (v) ) - A(v + E\mathcal G (v))[ E\mathcal G(v)] \\ &\quad\quad\quad - E( \mathcal G' ( v + E\mathcal G (v) )  [ F(v) - A(v)v ] ),\\
\text{and } v_0 &:= h_0 - E \mathcal G (h_0).
\end{align}
It is then clear that the linear boundary condition $(7.5)_2$ can be absorbed into the domain of the operator. Indeed, define $\tilde A_{\mathcal B} \in C^1 (\mathcal O ; \mathcal B(X_1; X_0))$ by means of $ \tilde A_{\mathcal B} (f) [g] :=  \tilde A(f) [g]$, where the domain of the linear operator $\tilde A_{\mathcal B} (f)$ is given by
\begin{equation}
D(\tilde A_{\mathcal B} (f) ) := X_1 \cap \{ g \in X_1 : \mathcal Bg = 0 \}.
\end{equation}
Then (7.3) is finally equivalent to the abstract evolution problem
\begin{equation} \label{34987345h3kj4h534534543A} \tag{7.6}
\partial_t v(t) +  \tilde A_{\mathcal B} (v(t))[v(t)] = \tilde F(v(t)), \; t >0, \quad v(0) = v_0,
\end{equation}
whenever $v(t) \in \mathcal O$ for $t > 0$.

We shall now analyse the transform in dependent variables $[h \mapsto v]$, given by 
\begin{equation} \label{27834637dsfdsf4634} \tag{7.7}
v(h) := (I - E\mathcal G) (h) = (I + E \mathcal G)^{-1}(h).
\end{equation}
\begin{lemma} \label{8943507634537456893}
There is a small neighbourhood $\mathcal O^\prime$ around zero in $X_\gamma$, such that $[h \mapsto v]$ is bijective. This bijection maps zero to zero and is also a local isomorphism around zero between the equilibrium sets of the equations for $h$ and $v$, corresponding to the equations (7.3) and (7.6), respectively.
\end{lemma}
\begin{proof}
By (7.7), the operator $[h \mapsto (I - E \mathcal G)h]$ is invertible as a mapping from $X_\gamma$ to $X_\gamma$. Furthermore, it maps zero to itself.

Pick some $v$ in the equilbrium set of (7.6). Then $\mathcal B v = 0$ and $\tilde A (v) v = \tilde F(v)$. Hence
\begin{align}
A (v + E\mathcal G (v) ) [ v+ E\mathcal G (v) ] = F(v + E\mathcal G (v) ) - E( \mathcal G' ( v + E\mathcal G (v) )  [ F(v) - A(v)v ] ).
\end{align}
Define $\rho := v + E\mathcal G v$. Then, since $E\mathcal G v$ vanishes on the boundary $\partial \Sigma$,
\begin{align} \label{34095649856gfdg45} \tag{7.8}
A (\rho ) \rho - F(\rho) =  - E( \mathcal G' ( \rho )  [ F(\rho) - A(\rho)\rho ] ).
\end{align}
Now $\rho$ is small whenever $v$ is small, hence $ I + E \mathcal G^\prime (\rho)$ is a bounded, linear, and invertible operator if $v$ is small in $X_\gamma$-norm. Hence (7.8) implies $\rho$ satisfies $A (\rho ) \rho = F(\rho)$. Since $\mathcal B v = 0$, we have $\mathcal B\rho = \mathcal G(\rho)$. This shows that $\rho$ is a stationary solution to (7.3). The proof is complete.
\end{proof}
The next theorem states conditions in which the set of equilibria $\mathcal E_h$ is a one-dimensional $C^1$-manifold in $X_1$, at least locally around zero. In such a case we can parametrize $\mathcal E_h$ over the kernel of the linearization, again at least locally around zero. 
We refer to \cite{garckeitokosaka} and \cite{MR2438774} for similiar considerations regarding surface diffusion flow.
\begin{theorem} \label{8495674596094564589645645089645}
Let $p \in (6,\infty)$, $q \in (19/10,2) \cap (2p/(p+1), 2)$, and $v_* = 0$. Let $X_c := X_\gamma \cap \{ u : u = const. \}$, 
\begin{equation}
{X}_\gamma^0 := X_\gamma \cap \{ u : \int_\Sigma u dx = 0 \},
\end{equation}
and
\begin{equation}
Z := Z_1 \times \R \times \R, \quad Z_1 := B^{2-1/q-3/p}_{qp}(\Sigma) \cap \{ u : \int_\Sigma u dx = 0 \}.
\end{equation}
Define the functional $I_* :  X_\gamma^0  \pfeil \R$ by
\begin{equation}
I_* (u ) := \int_\Sigma | \partial_x u |^2 dx - \omega_+ |u(L)|^2 - \omega_- |u(0)|^2,
\end{equation}
 and the nonlinear function $ F: X_c \times  X_\gamma^0 \pfeil Z$ by means of
\begin{equation}
F(m,u) := \left( \kappa(\Gamma) - \frac{1}{|\Gamma|} \int_\Gamma \kappa(\Gamma),\angle(\partial\Omega, \Gamma)_+ - \pi/2, \angle(\partial\Omega, \Gamma)_- - \pi/2 \right).
\end{equation}
Hereby, $\Gamma$ denotes the curve parametrized by $m + u \in X_\gamma$, $ \kappa(\Gamma)$ the curvature computed for that curve, and $\angle(\partial\Omega, \Gamma)_\pm$ denote the angles of $\Gamma$ with the outer boundary at the two contact points, whereby $+$ refers to the right contact point $x = L$.

Assume that $I_*$ is positive on $ X_\gamma^0$, that is, $I_* (u) \geq 0$ for all $u \in  X_\gamma^0$ and $I_* (u) = 0$, if and only if $ u = 0$.

Then $D_u F(0,0)$ is invertible, and there exists a small neighbourhood $U_c$ of zero in $X_c$, such that the set of of solutions to $F(m,u) = 0$ with $m \in U_c$ can be parametrized by a $C^1$-function $\mathsf u \in C^1(U_c ;  X_\gamma^0 )$, that is, $F(m, u) = 0$, $m \in U_c$, if and only if $u = \mathsf u(m)$.
\end{theorem}
\begin{proof}
Note that $F(0,0) = 0$. As in Theorem 5.2 in \cite{MR2438774} we see that the derivative at zero is given by
\begin{equation}
D_u F (0,0) v = ( \partial_x \partial_x v - P_0 \partial_x \partial_x v, (\partial_x  + \omega_+)v(L), (\partial_x  - \omega_-)v(0) ), \quad v \in \mathring X_\gamma.
\end{equation}
Here, $P_0 u := \frac{1}{|\Sigma|} \int_\Sigma u dx$ denotes the mean value of $u \in L_2(\Sigma)$.
We now show that $D_u F(0,0) \in \mathcal B( X_\gamma^0 ; Z)$ is invertible. For given $g = (g_1,g_2,g_3) \in Z$ we first reduce to the case $ g=(g_1,0,0)$ by solving auxiliary problems first. Define the linear operator $ A_0 : D(A_0) \subset X_\gamma^0 \pfeil Z_1 $ by means of $A_0 v := \partial_x \partial_x v - P_0 \partial_x \partial_x v$, with domain
\begin{equation} \label{45678945698745697456} \tag{7.9}
D(A_0) :=  X_\gamma^0 \cap \{ v : (\partial_x  + \omega_+)v(L) =(\partial_x  - \omega_-)v(0) = 0 \}.
\end{equation}
It is now enough to show that $A_0 \in \mathcal B ( D(A_0 ) ; Z_1 )$ is invertible. Since the embedding $D(A_0) \into Z_1$ is compact, the resolvent of $A_0$ is compact and the spectrum of $A_0$ only consists of isolated eigenvalues. Consider some eigenvalue $\lambda \in \mathbb C$ with corresponding eigenfunction $u \in D(A_0)$. Testing the resolvent equation $\lambda u = A_0 u$ with $u$ in $L_2(\Sigma)$, we obtain
\begin{equation}
\lambda |u|_2^2 = \int_0^L (\partial_x \partial_x u )\bar u dx - \frac{1}{L} \int_0^L \partial_x \partial_x u dx \int_0^L \bar u dx = \int_0^L (\partial_x \partial_x u ) \bar u dx,
\end{equation}
since $u$ is mean value free. By an integration by parts invoking the boundary conditions of (7.9) we see
\begin{equation}
\lambda |u|_2^2 = - \int_0^L | \partial_x u|^2 dx + \omega_+ |u(L)|^2 + \omega_- |u(0)|^2 = - I_* (u).
\end{equation}
Since $I_*$ is positive on $ X_\gamma^0$, $\lambda$ is real and strictly less than zero, provided $u \not = 0$. Hence zero belongs to the resolvent set of $A_0$. The rest of the claim now follows from the implicit function theorem, cf. \cite{MR2438774}, \cite{MR816732}.
\end{proof}
\begin{figure}[ht]
\begin{tikzpicture}
\draw (5,0) arc (270:90:2cm);
\draw (-5,0) arc (-90:90:2cm);
\draw (-3,2) -- (3,2);
\draw (4,0.8) node {$\partial\Omega$};
\draw (0.1,1.75) node {$\Sigma_\ast$};
\draw (1.5,3.05) node {$\Gamma_{m + \mathsf u(m)}$};
\draw (-3.04,2.3) arc (112.5:67.82:8cm);
\end{tikzpicture}
\caption{Parametrization by $[m \mapsto m + \mathsf u(m)]$ in Theorem \ref{8495674596094564589645645089645}.}
\label{figucxyfghfghcvgg}
\end{figure}
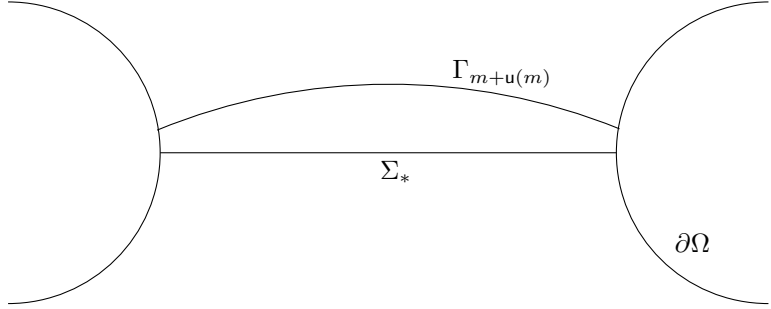
\begin{remark}
For situations in which $I_*$ is positive on $X_\gamma^0$,  we refer to the arguments of Section 4.1.
\end{remark}

We now apply the generalized principle of linearized stability \cite{pruessmulsek} to the nonlinear evolution problem (7.6).
\begin{theorem} \label{3498574958603476503475}
Let $p \in (6,\infty)$, $q \in (19/10,2) \cap (2p/(p+1), 2)$, $v_* = 0$, and $I_*$ positive on $X_\gamma^0$.
Let $\mathcal E_v$ be the set of equilibrium solutions to (7.6) and $U \subset X_\gamma$ a sufficiently small neighbourhood of zero in $X_\gamma$.
Then we have that $(\tilde A_{\mathcal B}, \tilde F) \in C^1 (U ; \mathcal B(X_1; X_0) \times X_0)$ and $\tilde A_{\mathcal B}(v_*)$ has maximal $L_p$-regularity with respect to $X_0$. Set $A_0 := \tilde A_{\mathcal B} (0)$, the linearization at $v_*$. Then $v_* = 0$ is normally stable, that is,
\begin{enumerate}
\item near $v_* = 0$, the set of equilibria $\mathcal E_v$ is a $C^1$-manifold in $X_1$ of dimension $1$,
\item the tangent space of $\mathcal E_v$ at $v_* = 0$ is isomorphic to the kernel of $A_0$, 
\item zero is a semi-simple eigenvalue of $A_0$, 
\item the spectrum $\sigma (A_0)$ satisfies $\sigma(A_0) \backslash \{ 0 \} \subset \mathbb C_+ := \{ \operatorname{Re} > 0 \} $.
\end{enumerate}
In particular, $v_* = 0$ is stable in $X_\gamma$ and there exists some $\delta > 0$, such that the unique solution
$v$ to the initial value $v_0 \in X_\gamma$ with $|v_0|_{X_\gamma} \leq \delta$ exists on the half line,
\begin{equation} \label{3498534958dsf634075} \tag{7.10}
v \in W^1_p(\R_+ ; X_0) \cap L_p(\R_+ ; D( \tilde A_{\mathcal B} ) ),
\end{equation}
 and $v(t)$ converges to some $v_\infty \in \mathcal E_v$ in $X_\gamma$ at an exponential rate as $t \pfeil \infty$.
\end{theorem}
\begin{remark}
Note that the solution $v$ in (7.10) naturally satisfies $\mathcal B v(t) = 0$ for all $t \in \R_+$.
\end{remark}
\begin{proof}
Let $U \subset X_\gamma$ be a sufficiently small neighbourhood of zero in $X_\gamma$. Then by (7.4), $(\tilde A_{\mathcal B}, \tilde F) \in C^1 (U ; \mathcal B(X_1; X_0) \times X_0)$. We have already shown (3) and (4) and that $A_0$ has maximal regularity. Let us show (1) and (2). 


Note that Theorem 7.2 gives a characterization of the equilibrium set $\mathcal E_h$ around zero. Indeed, the function $\mathsf u$ of Theorem 7.2 induces a $C^1$-parametrization of $\mathcal E_h$, locally around zero, over a one-dimensional parameter family. Moreover, $\mathsf u (0) = 0$.

Note that we have already shown that the kernel of $A_0$ consists of a one-parameter family of parabolas, whence we can, at least locally around zero, parametrize $\mathcal E_v$ over the kernel of $A_0$. 

The rest of the claim follows from \cite{pruessmulsek}.
\end{proof}
We then obtain a result for the solution $h$ of (7.3) in a natural way.
\begin{theorem}
The trivial equilibrium $h_* = 0$ of (7.3) is stable in $X_\gamma$, and there exists some $\delta_1 > 0$, such that the unique solution
$h$ to the initial value $h_0 \in X_\gamma$ with $|h_0|_{X_\gamma} \leq \delta_1$ exists on the half line $\R_+$,
 and $h(t)$ converges to some $h_\infty \in \mathcal E_h$ in $X_\gamma$ at an exponential rate as $t \pfeil \infty$.
\end{theorem}
\begin{proof}
This is a direct combination of Theorem 7.4 together with Lemma 7.1.
\end{proof}

\section*{Acknowledgements} The work of H.G. is supported by the DFG project RTG 2339. The work of M.R. is financially supported by the DFG project RTG 1692. The support is gratefully acknowledged.
\bibliographystyle{plain}
\bibliography{bibo}
\end{document}